\documentclass{article}

\usepackage{latexsym}
\usepackage{amsmath}
\usepackage{amsfonts}
\usepackage{fancyhdr}
\usepackage{graphicx}
\usepackage{dsfont}
\usepackage{color}
\usepackage{times}
\usepackage{hyperref}
\newcommand{\lag}{\mathcal{L}}

\newcommand{\qu}{\mathrm{q}}
\newcommand{\pe}{\mathrm{p}}

\usepackage{theorem}

\theoremheaderfont{\scshape}
\newtheorem{theorem}{Theorem}[section]
\newtheorem{corollary}[theorem]{Corollary}
\newtheorem{lemma}[theorem]{Lemma}
\newtheorem{prop1}[theorem]{Proposition}
\theorembodyfont{\normalfont}
\newtheorem{remark}{Remark}[section]
\newtheorem{definition}{Definition}[section]
\newenvironment{proof}{{ Proof. }}{\hbox{~} \hfill \rule{0.5em}{0.5em}\\}
	\numberwithin{equation}{section}

\begin{document}
		\title{ Topological Derivative for Shallow Water Equations
}         
\date{} 

\maketitle
\centerline{\scshape Mame Gor Ngom \footnote{mamegor.ngom@uadb.edu.sn}}
\medskip
{\footnotesize
\centerline{Universit\'e Alioune Diop de Bambey, }
\centerline{  Ecole Doctorale des Sciences et Techniques et Sciences de la Soci\'et\'e. }
\centerline{Laboratoire de Math\'ematiques de la D\'ecision et d'Analyse Num\'erique}
\centerline{ (L.M.D.A.N) F.A.S.E.G)/F.S.T. }
}

\medskip
\centerline{\scshape Ibrahima Faye \footnote{ibrahima.faye@uadb.edu.sn},  Diaraf Seck \footnote{diaraf.seck@ucad.edu.sn }}
\medskip
{\footnotesize
\centerline{Universit\'e Alioune Diop de Bambey, UFR S.A.T.I.C, BP 30 Bambey (S\'en\'egal),}
\centerline{  Ecole Doctorale des Sciences et Techniques et Sciences de la Soci\' et\'e. }
\centerline{Laboratoire de Math\'ematiques de la D\'ecision et d'Analyse Num\'erique}
\centerline{ (L.M.D.A.N). }
} 

\pagestyle{myheadings}
\renewcommand{\sectionmark}[1]{\markboth{#1}{}}
\renewcommand{\sectionmark}[1]{\markright{\thesection\ #1}}
\begin{abstract}\noindent 
Coastal erosion is a major and growing environmental problem that describes the movement of sand caused by tides, waves or currents. Several phenomena contribute to the significant advance of the sea. These include climate change, with rising sea levels due to the melting of ice at the Earth's poles, the amplification of the tidal effect, leading to the transport of large masses of sand, storms, etc. We contribute to this problem by using topological shape optimization techniques applied to a PDE model describing coastal erosion. We use shallow water equations as a model.		
\end{abstract}

\section{Introduction}
Coastal erosion is a phenomenon that has existed since time immemorial. It is defined as the loss of material due to the advance of the sea along the coast. There are many causes, but in general we can point to an increase in ocean height, and an acceleration in the transport of sand from beaches to oceans by increasingly powerful waves. For several decades now, scientists have been studying these phenomena, and a number of questions have arisen concerning coastal erosion: how can we improve the devices built to mitigate coastal erosion by acting on their shape? and where should they be installed?  Engineers, for their part, generally recommend the installation of dikes and dams to prevent or slow down coastal erosion phenomena, based on a study of the minimum action linked to the geometry of the fields. All these efforts produce interesting results, but they do not fully satisfy the community interested in these issues. For us, the study of topological or geometrical analysis tools of optimal shapes, and in particular the topological shape optimization for partial differential equations (PDE) modeling coastal erosion dynamics is interesting. We limit ourselves here to the topological shape optimization by focusing on the theoretical aspects. Coastal erosion can be modeled by several types of partial differential equations.  However we prefer to use the shallow water equations (SWEs) also known as Saint Venant equation. SWEs are a system of partial differential equations derived from the Navier-Stokes equations, that describe the motion of Newtonian fluids. They are used in many domains. For example, in river hydraulics, the two-dimensional SWEs can be used as a river model. This was studied by M. Honnorat in his PhD thesis \cite{honnorat2007assimilation}. As a reminder, river hydraulics is a science that enables the understanding and development of natural or artificial free-surface channels \cite{gharbi2016etudes}. Recently, L. Schlegel and V. Schulz have done interesting work in \cite{lukaswe} on shape optimization for the mitigation of coastal erosion using shallow water equations. Also SWEs are often used to model the flow of oceanographic and atmospheric fluids. Models of such systems can be used to predict areas potentially affected by pollution, coastal erosion, melting polar ice caps, etc. In the mathematical model, the flow is represented by a state variable which is the solution of the SWEs. Therefore the main objective of this paper is to provide theoretical results for optimality conditions of shape functionals under the constraints of the shallow water equation. To do this, we will rely mainly on the work of Delfour \cite{Delfour1} and Gangl and Sturm \cite{ganlsimplified}\cite{ganglautomated}. These are very efficient methods for computing topological derivatives. And the advantage is that they can be applied to many classes of PDEs. The paper is organized as follows: in Section \ref{sectpdecons}, we formulate the PDE-constrained optimization problem. In Section \ref{sectderivation}, we derive the ingredients needed to provide solutions to this problem. Section \ref{sectdertopo} presents the computation of the topological derivative, the derivation of the adjoint equations with the min-max approach. The last part, Section \ref{secttopogangl}, will otherwise give the results of topological derivative in the case of a shallow water equation including a viscosity using the Gangl and Sturm method \cite{ganglautomated}.
\section{ PDE-constrained topology optimization}\label{sectpdecons}
The shallow water model describes the evolution of a fluid by means of three equations: the continuity equation which translates the conservation of mass and the two momentum equations. These equations are obtained by integrating the Navier-Stokes equations over the vertical or simply over the entire depth with a hydrostatic approximation, i.e. neglecting the vertical acceleration. This approximation is necessary to translate the pressure into water height. However there are several variants of this model and we limit ourselves to the following formulation \cite{Bouh-Biza}:
\begin{eqnarray}\label{equaSV}
\frac{\partial U}{\partial t}+ F(U,\psi) = 0\;\;\;\Omega\times[0,T]
\end{eqnarray}
$U:\Omega\times [0,T[\rightarrow \mathbb{R}^3$ and $F$ are respectively given by
\begin{eqnarray}\label{Grandsyst}
U=(h,hu,hv)\;\;\;\text{and}\;\;\;F(U,\psi)=\left(\begin{array}{ccc}
	\frac{\partial (hu)}{\partial x}+\frac{\partial (hv)}{\partial y}\\ \frac{\partial (hu^2+\frac{g}{2}h^2)}{\partial x}+\frac{\partial (huv)}{\partial y}+gh\frac{\partial\psi}{\partial x}\\ \frac{\partial (huv)}{\partial x}+\frac{\partial (hv^2+\frac{g}{2}h^2)}{\partial y}+gh\frac{\partial\psi}{\partial y}
\end{array}\right)
\end{eqnarray} 
with initial and boundary conditions that we will define. Here $\Omega$ is a regular open domain in $\mathbb{R}^2$, $t$ is the time variable and $x$ and $y$ are the two-dimensional space variables. $u,\,v:\Omega\times[0,T[\rightarrow \mathbb{R}^+$ respectively represent the vertical and horizontal velocity of the water, $h$ is the height of the water and $g$ is the gravity and $\psi:\mathbb{R}^2\rightarrow \mathbb{R}$  is the topography of the bottom. The effects of bottom friction and the frictional forces of atmospheric pressure are neglected. The PDE-constrained topology optimization problem reads:
\begin{eqnarray*}
\min_\Omega 	&\;&J(\Omega)=\int_0^T\int_\Omega\vert\nabla U\vert^2\;dxdt+\int_0^T\int_\Omega\vert U-U_d\vert^2\;dxdt,\\[0.2cm]
&\,&s.t.\qquad\frac{\partial U}{\partial t}+F(U,\psi)=0\;\;\Omega\times[0,T]
\end{eqnarray*} 
The initial and boundary conditions are defined in the section \ref{sectdertopo} after the variational formulation of the problem is performed. In the next section, we will detail some of the ingredients we need to perform topological derivative calculations using the min-max approach.
\section{Derivation of the topological derivative}\label{sectderivation}
\subsection{Notations and definitions}
\begin{definition}
A Lagrangian function is a function of the form
\begin{eqnarray*}
	(s,x,y)\mapsto \lag(s,x,y):[0,\tau]\times X\times Y\rightarrow\mathbb{R}\;\;\;\;\tau >0
\end{eqnarray*}
where $X$ is a vector espace, $Y$ a non empty subset of vector space and the function $y\mapsto \lag(s,x,y)$ is affine.
\end{definition}
Associate with the parameter $s$ the parametrized minimax
\begin{eqnarray*}
s\mapsto \textrm{g}(s)=\inf_{x\in X}\sup_{y\in Y}\lag(s,x,y):[0,\tau]\rightarrow\mathbb{R}\;\;\;\text{and} \;\;\textrm{dg}(0)=\lim_{s\to 0^+}\frac{\textrm{g}(s)-\textrm{g}(0)}{s}.
\end{eqnarray*}
When the limits exist, we will use the following notations
\begin{eqnarray*}
d_s\lag(0,x,y)=\lim_{s \to 0^+}\frac{\lag(s,x,y)-\lag(0,x,y)}{s}\\\varphi\in X,\;\;d_x\lag(s,x,y;\varphi)=\lim_{\theta \to 0^+}\frac{\lag(s,x+\theta\varphi,y)-\lag(s,x,y)}{\theta}\\\phi\in Y\;\;d_y\lag(s,x,y;\phi)=\lim_{\theta \to 0^+}\frac{\lag(s,x,y+\theta\phi)-\lag(s,x,y)}{\theta}.
\end{eqnarray*}
Since $\lag(s,x,y)$ is affine en $y$, for all $(s,x)\in [0,\tau]\times X$,
\begin{eqnarray}
\forall\;y,\psi\in Y\;\;d_y\lag(s,x,y;\psi)=\lag(s,x,\psi)-\lag(s,x,0)=d_y\lag(s,x,0,\psi).
\end{eqnarray}
The state equation at $s\geq 0$
\begin{eqnarray}
\text{Find}\;x^s\in X\;\;\text{such that for all}\;\;\psi\in Y,\;\;d_y\lag(s,x^s,0;\psi)=0.
\end{eqnarray}
The set of states $x^s$ at $s\geq 0$ is denoted
\begin{eqnarray}
E(s)=\left\{ x^s\in X,\;\;\forall\;\psi\in Y,\;\;d_y\lag(s,x^s,0;\psi)=0 \right\}.
\end{eqnarray}
The adjoint equation at $s\geq 0$  is
\begin{eqnarray}
\text{Find}\;p^s\in Y\;\;\text{such that for all}\;\;\varphi\in X,\;\;d_x\lag(s,x^s,p^s;\varphi)=0.
\end{eqnarray}
The set of solutions $p^s$ at $s\geq 0$ is denoted
\begin{eqnarray}
Y(s,x^s)=\left\{ p^s\in Y,\;\;\forall\;\varphi\in X,\;\;d_x\lag(s,x^s,p^s;\varphi)=0 \right\}.
\end{eqnarray}
Finally the set of minimisers for the minimax is given by
\begin{eqnarray}
X(s)=\left\{ x^s\in X,\;\,\textrm{g}(s)=\inf_{x\in X}\sup_{y\in Y}\lag(s,x,y)=\sup_{y\in Y}\lag(s,x^s,y) \right\}.
\end{eqnarray}
\begin{lemma} \textnormal{\textbf{(Constrained infimum and minimax)}}\\
We have the following assertions
\begin{itemize}
	\item[(i)] 
	\begin{align*}
		\inf_{x\in X}\sup_{y\in Y}\lag(s,x,y)=\inf_{x\in E(s)}\lag(s,x,0)
	\end{align*}
	\item[(ii)] The minimax $\textrm{g}(s)=+\infty$ if and only if $E(s)=\emptyset$. And in this case we have $X(s)=X$.
	\item[(iii)] If $E(s)\neq\emptyset$, then 
	\begin{align*}
		X(s)=\left\{x^s\in E(s):\;\;\lag(s,x^s,0)=\inf_{x\in E(s)}\lag(s,x,0) \right\}\subset E(s)
	\end{align*}
	and $\textrm{g}(s)<+\infty$.
\end{itemize}
\end{lemma}
\begin{proof}
See \cite{Delfour1}.
\end{proof}

\noindent
To end this subsection, we give definitions and theorems on d-dimensional Minkowski content and d-rectifiability. 
\begin{definition}
Let $E$ be a subset of a metric space $X$. $E\subset X$ is $d$-rectifiable if it is the image of a compact subset $K$ of $\mathbb{R}^d$ by a continuous lipschitzian function $f:\mathbb{R}^d\rightarrow X$.
\end{definition}
Let $E$ be a closed compact set of $\mathbb{R}^N$ and $r\geq 0$, the distance function $d_E$ and the $r$-dilatation $E_r$ of $E$ are defined as follows:
\begin{eqnarray*}
d_E(x)=\inf_{x_0\in E}\lvert x-x_0\rvert,\;\;\;E_r=\{ x\in\mathbb{R}^N:\;d_E(x)\leq r\}.
\end{eqnarray*}
\begin{definition}
Given $d$, $0\leq d\leq N$ the upper and lower $d$-dimensional Minkowski contents of a set $E$ are defined by an $r$-dilatation of this set as follows
\begin{eqnarray*}
	M^{*d}(E)=\limsup_{r\to 0^+}\frac{m_N(E_r)}{\alpha_{N-d}r^{N-d}};\;\;\;M^{d}_*(E)=\liminf_{r\to 0^+}\frac{m_N(E_r)}{\alpha_{N-d}r^{N-d}}
\end{eqnarray*}
where $m_N$ is the Lebesgue measure in $\mathbb{R}^N$ and $\alpha_{N-d}$ is the volume of the ball of radius $1$ in $\mathbb{R}^{N-d}$.
\end{definition}
Both concepts can be found in \cite{Delfour1}.
\subsection{Some preliminary results}
We need the following assumption for everything that follows:\\
\textbf{Hypothesis (H0)}\\ Let $X$ be a vector space.
\begin{itemize}
\item[(i)]: For all $s\in [0,\tau]$, $x^0\in X(0)$, $x^s\in X(s)$ and $y\in Y$, the function $\theta\mapsto \lag(s,x^0+\theta(x^s-x^0),y):[0,1]\rightarrow\mathbb{R}$ is absolutely continuous. This implies that for almost all $\theta$ the derivative exists and is equal to
$d_x \lag(s,x^0+\theta(x^s-x^0),y;x^s-x^0)$ and it is the integral of its derivative. In particular
\begin{eqnarray*}
	\lag(s,x^s,y)=\lag(s,x^0,y)+\int_0^1d_x \lag(s,x^0+\theta(x^s-x^0),y;x^s-x^0)\;d\theta.
\end{eqnarray*} 
\item[ii)]: For all $s\in [0,\tau]$, $x^0\in X(0)$, $x^s\in X(s)$ and $y\in Y$, $\phi\in X$ and for almost all $\theta\in[0,1]$, $d_x \lag(s,x^0+\theta(x^s-x^0),y;\phi)$ exist et the functions $\theta\mapsto d_x\lag(s,x^0+\theta(x^s-x^0),y;\phi)$ belong to $L^1[0,1]$
\end{itemize}				
\begin{definition} (Sturm \cite{These-Sturm,Sturm2})
Given $x^0\in X(0)$ and $x^s\in X(s)$, the averaged adjoint equation is:
\begin{eqnarray*}
	\text{Find} \;\;y^s\in Y\;\;\forall\;\phi\in X,\;\;\int_0^1d_x \lag(s,x^0+\theta(x^s-x^0),y;\phi)\;d\theta=0
\end{eqnarray*}
and the set of solutions is noted $Y(s,x^0,x^s)$.\\
$Y(0,x^0,x^0)$ clearly reduces to the set of standard adjoint states $Y(0,x^0)$ at $s= 0$.
\end{definition}

\begin{theorem}\label{deltheosalwater}
Consider the Lagrangian functional
\begin{eqnarray*}
	(s,x,y)\mapsto \lag(s,x,y): [0,\tau]\times X\times Y\rightarrow\mathbb{R},\;\;\tau>0
\end{eqnarray*}
where $X$ and $Y$ are vector spaces and the function $y\mapsto \lag(s,x,y)$ is affine. Assume that $\textbf{(H0)}$ and the following hypotheses are satisfied\\
$\textbf{(H1)}$ for all $s\in [0,\tau]$, $\textrm{g}(s)$ is finite, $X(s)=\{x^s\}$ and $Y(0,x^0)=\{p^0\}$ are singletons,\\
$\textbf{(H2)}$ $d_s\lag(0,x^0,y^0)$  exists,\\
$\textbf{(H3)}$ The following limit exists
\begin{eqnarray*}
	R(x^0,y^0)=\lim_{s\to 0^+}\int_0^1 d_x\lag\left( s, x^0+\theta(x^s-x^0),p^0;\frac{x^s-x^0}{s}\right) d\theta.
\end{eqnarray*}
Then, $\textrm{dg}(0)$ exists and $\textrm{dg}(0)=d_s\lag(0,x^0,p^0)+R(x^0,p^0)$.
\end{theorem}
\begin{proof}
Recall that $\textrm{g}(s)=\lag(s,x^s,y)$ and $\textrm{g}(0)=\lag(0,x^0,y)$ for each $y\in Y$, then for a standard adjoint state $p^0$ at $s=0$
\begin{eqnarray*}
	\textrm{g}(s)-\textrm{g}(0)=\lag(s,x^s,p^0)-\lag(s,x^0,p^0)+(\lag(s,x^0,p^0)-\lag(0,x^0,p^0)).
\end{eqnarray*}
Dividing by $s> 0$, we obtain
\begin{eqnarray*}
	\frac{\textrm{g}(s)-\textrm{g}(0)}{s}&=&\frac{\lag(s,x^s,p^0)-\lag(s,x^0,p^0)}{s}+\frac{\lag(s,x^0,p^0)-\lag(0,x^0,p^0)}{s}\\&=&\int_0^1 d_x\lag\left( s, x^0+\theta(x^s-x^0),p^0;\frac{x^s-x^0}{s}\right) d\theta\\&+&\frac{\lag(s,x^0,p^0)-\lag(0,x^0,p^0)}{s}.
\end{eqnarray*}
Going to the limit when $s$ goes to zero, we obtain
\begin{eqnarray*}
	\textrm{dg}(0)&=&\lim_{s\to 0^+}\int_0^1 d_x\lag\left( s, x^0+\theta(x^s-x^0),p^0;\frac{x^s-x^0}{s}\right) d\theta+d_s\lag(0,x^0,p^0)\\&=&d_s\lag(0,x^0,p^0)+R(x^0,p^0).
\end{eqnarray*}
\end{proof}
\begin{corollary}
Consider the Lagrangian functional
\begin{eqnarray*}
	(s,x,y)\mapsto \lag(s,x,y): [0,\tau]\times X\times Y\rightarrow\mathbb{R},\;\;\tau>0
\end{eqnarray*}
where $X$ and $Y$ are vector spaces and the function $y\mapsto \lag(s,x,y)$ is affine. Assume that $\textbf{(H0)}$ and the following assumptions are satisfied:\\
$\textbf{(H1a)}$ for all $s\in [0,\tau]$, $X(s)\neq\emptyset$,  $g(s)$ is finite, and for each $x\in X(0)$, $Y(0,x)\neq\emptyset$,\\
$\textbf{(H2a)}$ for all $x\in X(0)$ and $p\in Y(0,x)$ $d_s\lag(0,x,p)$  exists,\\
$\textbf{(H3a)}$ there exist $x^0\in X(0)$ and $p^0\in Y(0,x^0)$ such that the following limit exists
\begin{eqnarray*}
	R(x^0,p^0)=\lim_{s\to 0^+}\int_0^1 d_x\lag\left( s, x^0+\theta(x^s-x^0),p^0;\frac{x^s-x^0}{s}\right) d\theta.
\end{eqnarray*}
Then, $\textrm{dg}(0)$ exists and there exist $x^0\in X(0)$ and $p^0\in Y(0,x^0)$ such that $\textrm{dg}(0)=d_s\lag(0,x^0,p^0)+R(x^0,p^0)$.
\end{corollary}
First we recall the following definitions and theorems which will be useful in the following.
\begin{definition}
An internally regular measure is a (positive) measure $\nu$ defined on the Borelian tribe of a separate topological space $X$ which verifies the following property
\begin{eqnarray}
	\forall A\subset X, \;\;A\;\text{borelian},\;\nu(A)=\{\nu (K),\;\text{K compact}\;\subset A\}.
\end{eqnarray}
\end{definition}
\textbf{Maximal Hardy-Littlewood function}
\begin{definition}
The maximal Hardy-Littlewood function is an operator $M$ that associates to any locally integrable function $f$ on $\mathbb{R}^n$ a function $Mf$; this function $Mf$ defined on each point $x \in\mathbb{R}^n$ as the upper bound of the mean values of $\vert f\vert$ on balls centered at $x$.
\end{definition}
\textbf{Mathematical formulation}: To any locally integrable function $f\in L^1_{Loc}(\mathbb{R}^n)$, we can associate the maximal Hardy-Littlewood function $Mf:\mathbb{R}^n\rightarrow [0,+\infty]$ defined by
\begin{eqnarray*}
Mf(x)=\sup_{r\ge 0}\frac{1}{\lambda_n(B(x,r))}\int_{B(x,r)}\vert f(t)\vert \;d\lambda_n(t)
\end{eqnarray*}
where $B(x,r)$ denotes the $\mathbb{R}^n$ ball centered at $x$ and of radius $r$ and $\lambda_n$ denotes the Lebesgue measure on $\mathbb{R}^n$.
\\ \textsf{Maximal Hardy-Littlewood inequality:}
\begin{prop1}
For any integrable application $f$ on $\mathbb{R}^n$ and any real $c> 0$, we have
\begin{eqnarray*}
	\lambda_n\left([Mf\geq c]\right)\leq 3^n\frac{\Vert f\Vert_1} {c}.
\end{eqnarray*}
\end{prop1}
\begin{proof}
If you want to go to the limit when $d\rightarrow c^-$, just show that 
\begin{eqnarray*}
	\forall d\geq 0\;\; \lambda_n\left([Mf> d]\right)\leq 3^n\Vert f\Vert_1/d
\end{eqnarray*}
and for this by interior regularity, it suffices to show that for any compact $K$ included in $[Mf> d]$
\begin{eqnarray*}
	\lambda_n(K)\leq 3^n\Vert f\Vert_1/d.
\end{eqnarray*}
By compactness, $K$ is covered by a finite number of such balls, and we can choose among them balls $\left( B(x,r_x)\right)_{x\in X}$ disjointed such as $K\subset \cup_{x\in X}B(x,3r_x)$. Then we have
\begin{eqnarray*}
	\lambda_n(K)\leq \lambda_n\left(\cup_{x\in X}B(x,3r_x) \right)&\leq&\sum_{x\in X}\lambda_n\left(B(x,3r_x) \right)=3^n\sum_{x\in X}\lambda_n\left(B(x,r_x)\right)\\
	&\leq& \frac{3^n}{d}\sum_{x\in X}\int_{B(x,r_x)}\vert f(t)\vert d\lambda_n(t)\leq\frac{3^n\Vert f\Vert_1}{d}
\end{eqnarray*}
because the balls are disjointed.
\end{proof}
\textbf{Lebesgue differentiation theorem }
\begin{theorem} For any function integrable in the Lebesgue sense $f\in L^1(\mathbb{R}^n)$, we have for almost all $x\in \mathbb{R}^n$
\begin{eqnarray*}
	\lim_{r\to 0}\frac{1}{\lambda(B(x,r))}\int_{B(x,r)}\vert f(t)-f(x)\vert \;d\lambda(t)=0
\end{eqnarray*}
where $B(x,r)$ denotes the $\mathbb{R}^n$ ball centered at $x$ and of radius $r$ and $\lambda$ denotes the Lebesgue measure.
\end{theorem}
\begin{proof}
For $x\in\mathbb{R}^n$ and $r\ge 0$, we set
\begin{eqnarray*}
	T_r(f)(x)=\frac{1}{\lambda(B(x,r))}\int_{B(x,r)}\vert f(t)-f(x)\vert \;d\lambda(t)\\[0.2cm]
		\text{and}\;\;\;T(f)(x)=\limsup_{r\to 0} T_r(f)(x)
	\end{eqnarray*}
	We take here the upper limit because the limit when $r$ tends to $0$ is not necessarily defined. The idea is to show that $T(f)=0$ almost everywhere by showing that for all $c> 0$, $\{ Tf> 0\}$ is negligible.\\Let $k> 0$ be an integer. We know from the density of continuous functions in the spaces $L^p(\Omega)$ that there exists a continuous function $g$ such that $\Vert f-g\Vert_1< \frac{1}{k}$. If we put $h=f-g$, then we have $T_r(h)(x)\leq \frac{1}{\lambda( B(x,r))}\int_{B(x,r)}\vert h(t)\vert \;d\lambda(t)+\vert h(x)\vert$ and therefore $T_r(h)(x)\leq Mh(x)+\vert h(x)\vert$ where $Mh$ is the maximum Hardy-Littlewood function associated with $h$. The continuity of $g$ ensures that $T(g)=0$, from $f=g+h$ we derive $T_r(f)\leq T_r(g)+T_r(h)$ and thus passing to the upper limit $T(f)\leq T(g)+T(h)=T(h)$ which proves that $T(f)\leq Mh+\vert h\vert$.\\Then for all $c> 0$ we have the following inclusion $\{Tf>2c \}\subset \{Mh> c \}\cup \{\vert h\vert> c \}$. According to the Hardy-Littlewood inequality $\lambda\left(\{Mh> c \} \right)\leq \frac{3^n\Vert h\Vert_1}{c}\leq \frac{3^n}{ck}$ and on the other hand $\lambda\left( \{\vert h\vert> c \} \right)\leq \frac{\vert h\vert_1}{c}\leq \frac{1}{ck}$. The set $\{Mh> c \}\cup \{\vert h\vert> c \}$ which is measurable has therefore a smaller measure than $\frac{3^n+1}{ck}$. This means that $\{Tf> 2c \}$ is included in a set of measure less than $\frac{3^n+1}{ck}$ for all $k\ge 0$, by taking the intersection on $k\ge 0$ of all these sets we show that $\{Tf> 2c \}$ is included in a set of zero measure, $\{Tf\ge 2c \}$ is negligible.					
\end{proof}

\noindent
We state the following theorem which extends the Lebesgue differentiation theorem from $d=0$ to $0\leq d< N$.
\begin{theorem}\label{diffLbg}
	Let $E$ be a compact subset of $\mathbb{R}^N$ and $0\leq d< N$ an integer. Let $\alpha_k$ be the volume of the unit ball in $\mathbb{R}^k$. Assume that
	\begin{itemize}
		\item[(1)] $E$ is a $d$-rectifiable subset of $\mathbb{R}^N$ such that $\partial E=E$ and $0< H^d(E)<\infty$,
		\item[(2)] $E$ has a positive reach, i.e. there is $R>0$ such that $d^2_E\in C^{1,1}(E_R)$
		\item[(3)] and $f$ is continuous in $E_R$.\\Then
		\begin{eqnarray*}
			\lim_{r \to 0^+}\frac{1}{\alpha_{N-d}r^{N-d}}\int_{E_r}f\;d\xi=\lim_{r\to 0^+}\frac{1}{\alpha_{N-d}r^{N-d}}\int_{E_r}fop_E \;d\xi=\int_E f\;dH^d.
		\end{eqnarray*}
		where $H^d$ is the $d$-dimensional Hausdorff measure given in the following definition:
	\end{itemize}
	\begin{definition} (Hausdorff measure)\\
		Let $\Omega\subset\mathbb{R}^N$ and $d\in\mathbb{R}_+$. We define the $d$-dimensional Hausdorff measure of $A$ by 
		$$ H^d(A)=\liminf_{\epsilon \to 0}\left\{ \sum_{k=1}^\infty \alpha(d)r(C_k)^d;\;\;A\subset \cap_{k=1}^\infty(C_k)\;\;diam(C_k)\leq \epsilon \right\}$$
		where $C_k$ are arbitrary parts of $\mathbb{R}^N$, $r(C_k)=diam(C_k)/2$ is the half-diameter or radius of $C_k$ and $\alpha(d)=\frac{\pi^{d/2}}{\Gamma\left(\frac{d}{2}+1\right)}$  is the volume of the unit ball of dimension $d$ and the Euler function $\Gamma$ is defined by the following integral
		$$\Gamma(p)=\int_0^\infty \mathrm{e}^{-x}x^{p-1}\;dx.$$
	\end{definition}
\end{theorem}
\section{Topological derivative}\label{sectdertopo}

\subsection{Problem model}
Let $\Omega$ be a regular open domain in $\mathbb{R}^2$ and $U:\Omega\times [0,T]\rightarrow \mathbb{R}^3$ solution of (\ref{equaSV}). We recall the following notations. For a function $\varphi:\mathbb{R}^2\rightarrow\mathbb{R}^3$, we denote by $\nabla\varphi\in\mathbb{R}^{3\times2}$ its Jacobian, $(\nabla\varphi)_{i,k}=\frac{\partial\varphi_i}{\partial x_k}$ for $i=\{1,2,3\}$, $k\in\{1,2\}$. The Euclidean vector product $a\cdot b=\sum_{1}^{3}a_ib_i$ for $a,b\in\mathbb{R}^3$. $A:B=\sum_{i=1}^{3}\sum_{k=1}^{2}A_{i,k}B_{i,k}$ is the Frobenius inner product of two matrices $A,B\in\mathbb{R}^{3\times2}$.
We recall the functional  
\begin{eqnarray}\label{shfcgen}
	J(\Omega)=\int_0^T\int_\Omega\vert\nabla U\vert^2\;dxdt+\int_0^T\int_\Omega\vert U-U_d\vert^2\;dxdt,
\end{eqnarray} 
where $U_d:L^2(\Omega\times [0,T[)^3\rightarrow \mathbb{R}$ given by $U_d=(u_{1d},u_{2d},u_{3d})$ is a given function and $U$ is a solution of the equation (\ref{equaSV}) with the notation $\vert\nabla U\vert^2=\nabla U:\nabla U$. In what follows we set the change of variables $\qu_1=hu$ and $\qu_2=hv$ and write the system (\ref{equaSV}) as follows

\begin{eqnarray}\label{reecritureSW}
	\begin{cases}
		\frac{\partial h}{\partial t}+	\frac{\partial \qu_1}{\partial x}+\frac{\partial \qu_2}{\partial y}=0\\[0.2cm]\frac{\partial \qu_1}{\partial t}+\frac{\partial}{\partial x}\left(\frac{\qu_1^2}{h}+\frac{g}{2}h^2\right)+\frac{\partial}{\partial
			y}\left(\frac{\mathrm{q}_1\qu_2}{h}\right)+gh\frac{\partial\psi}{\partial x}=0\\[0.2cm] \frac{\partial \qu_2}{\partial t}+\frac{\partial}{\partial x}\left(\frac{\qu_1\qu_2}{h}\right)+\frac{\partial}{\partial y}\left(\frac{\qu_2^2}{h}+\frac{g}{2}h^2\right)+gh\frac{\partial\psi}{\partial y}=0.
	\end{cases}		
\end{eqnarray} 
We recall that in this system (\ref{reecritureSW}), the first equation represents the continuity equation and the last two represent the momentum equations. So to do the variational formulation of (\ref{reecritureSW}), first we choose an arbitrary function $\mathrm{P}=(\pe_1,\pe_2,\pe_3)\in H^1(\Omega\times[0,T])^3$. The continuity equation is multiplied by $\pe_1$ and the two momentum equations by $\pe_2$ and $\pe_3$ respectively. The sum of these two products is then integrated in space and time
\begin{eqnarray}
\begin{split}\label{1forvarSw}
	&\;\int_0^T\int_{\Omega}\left[	\frac{\partial h}{\partial t}+\frac{\partial \qu_1}{\partial x}+\frac{\partial \qu_2}{\partial y}\right]\pe_1\;dxdt+\\ &\;\int_0^T\int_{\Omega}\left[\frac{\partial \qu_1}{\partial t}+\frac{\partial}{\partial x}\left(\frac{\qu_1^2}{h}+\frac{g}{2}h^2\right)+\frac{\partial}{\partial
		y}\left(\frac{\mathrm{q}_1\qu_2}{h}\right)+gh\frac{\partial\psi}{\partial x}\right]\pe_2 \;dxdt+\\ &\;\int_0^T\int_{\Omega}\left[	\frac{\partial \qu_2}{\partial t}+\frac{\partial}{\partial x}\left(\frac{\qu_1\qu_2}{h}\right)+\frac{\partial}{\partial y}\left(\frac{\qu_2^2}{h}+\frac{g}{2}h^2\right)+gh\frac{\partial\psi}{\partial y}\right]\pe_3\;dxdt=0.
\end{split}
\end{eqnarray}
Let $n$ be the normal vector defined by $n=(n_x,n_y)$ where $n_x$ is the normal along the $x$ direction and $n_y$ is the normal along the $y$ direction. We impose the following initial and boundary conditions repectively
\begin{eqnarray}
h(x,0)=h_0(x),\;\;\qu_1(x,0)=\qu_{10}(x),\;\;\qu_2(x,0)&=&\qu_{20}(x) \qquad \text{in}\;\;\Omega\label{ICSWE}\\
\qu_1 n_x+\qu_2 n_y&=&0 \,\qquad\qquad\text{in}\;\;\partial\Omega\times [0,T].\label{BCSWE}
\end{eqnarray}
We now define $\mathcal{V}_{\Omega}$, the subspace of $H^1(\Omega\times [0,T])^3$ given by
\begin{eqnarray}\label{funcessaiSW}
\mathcal{V}_{\Omega}=\left\{U=(h,\qu_1,\qu_2)\in H^1(\Omega\times [0,T])^3: \;\; \qu_1 n_x+\qu_2 n_y=0\;\;\text{in}\;\;\partial\Omega\times [0,T] \right\}.
\end{eqnarray}
We are looking for a solution $U$ of (\ref{equaSV}) in this space $\mathcal{V}_{\Omega}$. So by integration by parts (\ref{1forvarSw}) with respect to time and space
\begin{eqnarray*}
&\;&\int_0^T\int_{\Omega}-\frac{\partial \pe_1}{\partial t}h\;dxdt+\int_{\Omega}\left[h(x,T)\pe_1(x,T)-h_0\pe_1(x,0)\right]\;dx+ \int_0^T\int_{\Omega}-\frac{\partial \pe_1}{\partial x}\qu_1\;dxdt+\\ 
&\;&\int_0^T\int_{\partial\Omega}\pe_1\qu_1 n_x\;d\sigma dt+\int_0^T\int_{\Omega}-\frac{\partial \pe_1}{\partial y}\qu_2\;dxdt+\int_0^T\int_{\partial\Omega} \pe_1\qu_2 n_y\;d\sigma dt+\\
&\;&\int_0^T\int_{\Omega}-\frac{\partial \pe_2}{\partial t}\qu_1\;dxdt+\int_{\Omega}\left[\qu_1(x,T)\pe_2(x,T)-\qu_{10}\pe_2(x,0)\right]\;dx+\\ 
&\;&\int_0^T\int_{\Omega}-\frac{\partial\pe_2}{\partial x}\left(\frac{\qu_1^2}{h}+\frac{g}{2}h^2\right) \;dxdt+\int_0^T\int_{\partial\Omega}\left(\frac{\qu_1^2}{h}+\frac{g}{2}h^2\right)\pe_2 n_x\;d\sigma dt+\\
&\;&\int_0^T\int_{\Omega}-\frac{\partial\pe_2}{\partial
	y}\left(\frac{\mathrm{q}_1\qu_2}{h}\right)\;dxdt+\int_0^T\int_{\partial\Omega}\left(\frac{\mathrm{q}_1\qu_2}{h}\right)\pe_2n_y\;d\sigma dt+\int_0^T\int_{\Omega}gh\frac{\partial\psi}{\partial x}\pe_2 \;dxdt+\\
&\;&\int_0^T\int_{\Omega}-\frac{\partial \pe_3}{\partial t}\qu_2\;dxdt+\int_{\Omega}\left[\qu_2(x,T)\pe_3(x,T)-\qu_{20}\pe_3(x,0)\right]\;dx+ \\
&\;&\int_0^T\int_{\Omega}-\frac{\partial\pe_3}{\partial
	x}\left(\frac{\mathrm{q}_1\qu_2}{h}\right)\;dxdt+\int_0^T\int_{\partial\Omega}\left(\frac{\mathrm{q}_1\qu_2}{h}\right)\pe_3n_x\;d\sigma dt+\int_0^T\int_{\Omega}gh\frac{\partial\psi}{\partial y}\pe_3 \;dxdt+\\	
&\;&\int_0^T\int_{\Omega}-\frac{\partial\pe_3}{\partial y}\left(\frac{\qu_2^2}{h}+\frac{g}{2}h^2\right) \;dxdt+\int_0^T\int_{\partial\Omega}\left(\frac{\qu_2^2}{h}+\frac{g}{2}h^2\right)\pe_3 n_y\;d\sigma dt=0.\\
\end{eqnarray*}
We choose the following boundary conditions for the test function $P$
\begin{align}
\label{condauborSW1} &\;\pe_2n_x+\pe_3n_y=0\;\;\;\text{in}\;\;\partial\Omega\times [0,T]\\\label{condauborSW2}
&\,\left(\frac{\qu_1^2}{h}+\frac{g}{2}h^2\right)\pe_2 n_x+\left(\frac{\mathrm{q}_1\qu_2}{h}\right)\pe_2n_y+\left(\frac{\mathrm{q}_1\qu_2}{h}\right)\pe_3n_x+\left(\frac{\qu_2^2}{h}+\frac{g}{2}h^2\right)\pe_3 n_y=0\;\;\;\text{in}\;\;\partial\Omega\times [0,T].
\end{align}
And we define the following space
\begin{eqnarray}\label{functestSW}
\mathcal{W}_{\Omega}=\left\{U=(\pe_1,\pe_2,\pe_3)\in H^1(\Omega\times [0,T])^3 \;\;\text{such that}\;\; (\ref{condauborSW1})\;\text{and }\;(\ref{condauborSW2})\;\;\text{are satisfied} \right\}.
\end{eqnarray}
We can say that for any $U=(h,\qu_1,\qu_2)\in\mathcal{V}_{\Omega}$ and $P=(\pe_1,\pe_2,\pe_3)\in\mathcal{W}_{\Omega}$
\begin{align*}
&\;\int_0^T\int_{\Omega}-\frac{\partial \pe_1}{\partial t}h\;dxdt+\int_{\Omega}\left[h(x,T)\pe_1(x,T)-h_0\pe_1(x,0)\right]\;dx+ \int_0^T\int_{\Omega}-\frac{\partial \pe_1}{\partial x}\qu_1\;dxdt+\\ 
&\;\int_0^T\int_{\Omega}-\frac{\partial \pe_1}{\partial y}\qu_2\;dxdt+\int_0^T\int_{\Omega}-\frac{\partial \pe_2}{\partial t}\qu_1\;dxdt+\int_{\Omega}\left[\qu_1(x,T)\pe_2(x,T)-\qu_{10}\pe_2(x,0)\right]\;dx+\\ 
&\;\int_0^T\int_{\Omega}-\frac{\partial\pe_2}{\partial x}\left(\frac{\qu_1^2}{h}+\frac{g}{2}h^2\right) \;dxdt+\int_0^T\int_{\Omega}-\frac{\partial\pe_2}{\partial
	y}\left(\frac{\mathrm{q}_1\qu_2}{h}\right)\;dxdt+\int_0^T\int_{\Omega}gh\frac{\partial\psi}{\partial x}\pe_2 \;dxdt+\\	
&\;\int_0^T\int_{\Omega}-\frac{\partial \pe_3}{\partial t}\qu_2\;dxdt+\int_{\Omega}\left[\qu_2(x,T)\pe_3(x,T)-\qu_{20}\pe_3(x,0)\right]\;dx+ \int_0^T\int_{\Omega}-\frac{\partial\pe_3}{\partial
	x}\left(\frac{\mathrm{q}_1\qu_2}{h}\right)\;dxdt+\\
&\;\int_0^T\int_{\Omega}gh\frac{\partial\psi}{\partial y}\pe_3 \;dxdt+\int_0^T\int_{\Omega}-\frac{\partial\pe_3}{\partial y}\left(\frac{\qu_2^2}{h}+\frac{g}{2}h^2\right) \;dxdt=0.\\
\end{align*}

With the change of variables $\qu_1=hu$ and $\qu_2=hv$, we redefine the objective function (\ref{shfcgen}) as follows:
\begin{align}\label{shfcgenredef}
J(\Omega)=\int_0^T\int_{\Omega}\vert\nabla h\vert^2+\vert\nabla \qu_1\vert^2+\vert\nabla\qu_2\vert^2\;dxdt+\int_0^T\int_{\Omega}\vert h-u_{1d}\vert^2+\vert \qu_1-u_{2d}\vert^2+\vert \qu_2-u_{3d}\vert^2\;dxdt.
\end{align}
\subsection{Perturbed problems}
\noindent
Let $E$ be a compact $d$-rectifiable subset of $\mathbb{R}^2$ such that its Hausdorff measure $H^d(E)$ is finite. Suppose that there exists $R>0$ such that $d^2_E\in C^{1,1}(E_R)$. We consider the $r$-dilatation $E_r$ of $E$ and $s$ an auxiliary variable given by $s=\alpha_{2-d}r^{2-d}$ where $\alpha_{2-d}$ is the volume of the unit ball of $\mathbb{R}^{2-d}$. We consider a perturbation of the form $\mathds{1}_{\Omega_s}=\mathds{1}_{\Omega}-\mathds{1}_{E_r}$ where $\mathds{1}_{A}=\chi_A$ denotes the indicator function of $A$. So the perturbed domain is defined by $s\mapsto \Omega_s=\Omega\backslash E_r$. Let $U^s=(h^s,\qu_1^s,\qu_2^s)$ be the solution of the system
\begin{eqnarray}
\frac{\partial U^s}{\partial t}+F(U^s,\psi)=0\;\;\text{in}\;\;\Omega_s\times[0,T].
\end{eqnarray}
The associated objective function is
\begin{eqnarray}\label{pershfcgenredef}
\begin{split}
	J(\Omega_s)=&\int_0^T\int_{\Omega_s}\vert\nabla h^s\vert^2+\vert\nabla \qu_1^s\vert^2+\vert\nabla\qu_2^s\vert^2\;dxdt+\\&\int_0^T\int_{\Omega_s}\vert h^s-u_{1d}\vert^2+\vert \qu_1^s-u_{2d}\vert^2+\vert \qu_2^s-u_{3d}\vert^2\;dxdt.
\end{split}
\end{eqnarray}
We define the perturbed spaces $\mathcal{V}_{\Omega_s}$ and $\mathcal{W}_{\Omega_s}$ with reference to (\ref{funcessaiSW}) and (\ref{functestSW}) when we replace $\Omega$ by $\Omega_s$. Then for $\varphi=(\varphi_1,\varphi_2,\varphi_3)\in \mathcal{V}_{\Omega_s}$ and $\Phi=(\phi_1,\phi_2,\phi_3)\in \mathcal{W}_{\Omega_s}$, the perturbed Lagrangian is define by
\begin{eqnarray*}
\lag(s,\varphi,\Phi)&=&\int_0^T\int_{\Omega_s}\vert\nabla \varphi_1\vert^2+\vert\nabla \varphi_2\vert^2+\vert\nabla\varphi_3\vert^2\;dxdt+\\&\;&\int_0^T\int_{\Omega_s}\vert \varphi_1-u_{1d}\vert^2+\vert\varphi_2-u_{2d}\vert^2+\vert \varphi_3-u_{13}\vert^2\;dxdt+\\
&\;&\int_0^T\int_{\Omega_s}-\frac{\partial \phi_1}{\partial t}\varphi_1\;dxdt+\int_{\Omega_s}\left[\varphi_1(x,T)\phi_1(x,T)-h_0\phi_1(x,0)\right]\;dx+\\ 
&\;& \int_0^T\int_{\Omega_s}-\frac{\partial \phi_1}{\partial x}\varphi_2\;dxdt+\int_0^T\int_{\Omega_s}-\frac{\partial \phi_1}{\partial y}\varphi_3\;dxdt+\int_0^T\int_{\Omega_s}-\frac{\partial \phi_2}{\partial t}\varphi_2\;dxdt+\\
&\;&\int_{\Omega_s}\left[\varphi_2(x,T)\phi_2(x,T)-\qu_{10}\phi_2(x,0)\right]\;dx+\int_0^T\int_{\Omega_s}-\frac{\partial\phi_2}{\partial x}\left(\frac{\varphi_2^2}{\varphi_1}+\frac{g}{2}\varphi_1^2\right) \;dxdt+\\ 
&\;&\int_0^T\int_{\Omega_s}-\frac{\partial\phi_2}{\partial
	y}\left(\frac{\varphi_2\varphi_3}{\varphi_1}\right)\;dxdt+\int_0^T\int_{\Omega_s}g\varphi_1\frac{\partial\psi}{\partial x}\phi_2 \;dxdt+\int_0^T\int_{\Omega_s}-\frac{\partial \phi_3}{\partial t}\varphi_3\;dxdt+\\	
&\;&\int_{\Omega_s}\left[\varphi_3(x,T)\phi_3(x,T)-\qu_{20}\phi_3(x,0)\right]\;dx+ \int_0^T\int_{\Omega_s}-\frac{\partial\phi_3}{\partial
	x}\left(\frac{\varphi_2\varphi_3}{\varphi_1}\right)\;dxdt+\\
&\;&\int_0^T\int_{\Omega_s}g\varphi_1\frac{\partial\psi}{\partial y}\phi_3 \;dxdt+\int_0^T\int_{\Omega_s}-\frac{\partial\phi_3}{\partial y}\left(\frac{\varphi_3^2}{\varphi_1}+\frac{g}{2}\varphi_1^2\right) \;dxdt.
\end{eqnarray*}
In this case, following \cite{TD1-Delfour}\cite{TD2} the objective functions are defined by
\begin{eqnarray*}
J(\Omega_s)&=&\min_{\varphi\in \mathcal{V}_{\Omega_s}}\max_{\Phi\in \mathcal{W}_{\Omega_s}}\lag(s,\varphi,\Phi)\\
J(\Omega)&=&\min_{\varphi\in \mathcal{V}_{\Omega}}\max_{\Phi\in \mathcal{W}_{\Omega}}\lag(0,\varphi,\Phi).
\end{eqnarray*}
For $s=\alpha_{2-d}r^{2-d}$, we define $\mathrm{g}(s)=J(\Omega_s)$ and $\mathrm{g}(0)=J(\Omega)$. Then the topological derivative of $J$ is given by
\begin{eqnarray}
DJ_T=\mathrm{dg}(0)=\lim_{s\searrow 0} \frac{\mathrm{g}(s)-\mathrm{g}(0)}{s}.
\end{eqnarray}
In what follows, we define the final conditions for any function $\Phi\in \mathcal{W}_{\Omega}$ by
\begin{eqnarray}
\Phi(x,T)=0\;\;\;\text{in}\;\;\Omega.
\end{eqnarray}
We proceed to a derivative of the perturbed Lagrangian with respect to the variables $s$, $\varphi$ and $\Phi$. For $d_s\lag(s,\varphi,\Phi)$, we calculate the difference $\lag(s,\varphi,\Phi)-\lag(0,\varphi,\Phi)$ where $\lag(0,\varphi,\Phi)$ is the unperturbed Lagrangian. Therefore
\begin{align*}
\lag(0,\varphi,\Phi)=&\;\int_0^T\int_{\Omega}\vert\nabla \varphi_1\vert^2+\vert\nabla \varphi_2\vert^2+\vert\nabla\varphi_3\vert^2\;dxdt+\\&\;\int_0^T\int_{\Omega}\vert \varphi_1-u_{1d}\vert^2+\vert\varphi_2-u_{2d}\vert^2+\vert \varphi_3-u_{3d}\vert^2\;dxdt+\\
&\;\int_0^T\int_{\Omega}-\frac{\partial \phi_1}{\partial t}\varphi_1\;dxdt-\int_{\Omega} h_0\phi_1(x,0)\;dx+\\ 
&\; \int_0^T\int_{\Omega}-\frac{\partial \phi_1}{\partial x}\varphi_2\;dxdt+\int_0^T\int_{\Omega}-\frac{\partial \phi_1}{\partial y}\varphi_3\;dxdt+\int_0^T\int_{\Omega}-\frac{\partial \phi_2}{\partial t}\varphi_2\;dxdt+\\
&\;\int_{\Omega}-\qu_{10}\phi_2(x,0)\;dx+\int_0^T\int_{\Omega}-\frac{\partial\phi_2}{\partial x}\left(\frac{\varphi_2^2}{\varphi_1}+\frac{g}{2}\varphi_1^2\right) \;dxdt+\\ 
&\;\int_0^T\int_{\Omega}-\frac{\partial\phi_2}{\partial
	y}\left(\frac{\varphi_2\varphi_3}{\varphi_1}\right)\;dxdt+\int_0^T\int_{\Omega}g\varphi_1\frac{\partial\psi}{\partial x}\phi_2 \;dxdt+\\	
&\;\int_0^T\int_{\Omega}-\frac{\partial \phi_3}{\partial t}\varphi_3\;dxdt+\int_{\Omega} -\qu_{20}\phi_3(x,0)\;dx+ \int_0^T\int_{\Omega}-\frac{\partial\phi_3}{\partial
	x}\left(\frac{\varphi_2\varphi_3}{\varphi_1}\right)\;dxdt+\\
&\;\int_0^T\int_{\Omega}g\varphi_1\frac{\partial\psi}{\partial y}\phi_3 \;dxdt+\int_0^T\int_{\Omega}-\frac{\partial\phi_3}{\partial y}\left(\frac{\varphi_3^2}{\varphi_1}+\frac{g}{2}\varphi_1^2\right) \;dxdt.
\end{align*}

\begin{align*}
\lag(s,\varphi,\Phi)-\lag(0,\varphi,\Phi)=&\;-\int_0^T\int_{E_r}\vert\nabla \varphi_1\vert^2+\vert\nabla \varphi_2\vert^2+\vert\nabla\varphi_3\vert^2\;dxdt\\&\;-\int_0^T\int_{E_r}\vert \varphi_1-u_{1d}\vert^2+\vert\varphi_2-u_{2d}\vert^2+\vert \varphi_3-u_{3d}\vert^2\;dxdt\\
&\; +\int_0^T\int_{E_r}\frac{\partial \phi_1}{\partial t}\varphi_1\;dxdt+\int_{E_r}h_0\phi_1(x,0)\;dx\\ 
&\; +\int_0^T\int_{E_r}\frac{\partial \phi_1}{\partial x}\varphi_2\;dxdt+\int_0^T\int_{E_r}\frac{\partial \phi_1}{\partial y}\varphi_3\;dxdt+\int_0^T\int_{E_r}\frac{\partial \phi_2}{\partial t}\varphi_2\;dxdt\\
&\;+\int_{E_r}\qu_{10}\phi_2(x,0)\;dx+\int_0^T\int_{E_r}\frac{\partial\phi_2}{\partial x}\left(\frac{\varphi_2^2}{\varphi_1}+\frac{g}{2}\varphi_1^2\right) \;dxdt\\ 
&\; +\int_0^T\int_{E_r}\frac{\partial\phi_2}{\partial
	y}\left(\frac{\varphi_2\varphi_3}{\varphi_1}\right)\;dxdt-\int_0^T\int_{E_r}g\varphi_1\frac{\partial\psi}{\partial x}\phi_2 \;dxdt\\	
&\; +\int_0^T\int_{E_r}\frac{\partial \phi_3}{\partial t}\varphi_3\;dxdt+\int_{E_r}\qu_{20}\phi_3(x,0)\;dx+\int_0^T\int_{E_r}\frac{\partial\phi_3}{\partial
	x}\left(\frac{\varphi_2\varphi_3}{\varphi_1}\right)\;dxdt\\
&\; -\int_0^T\int_{E_r}g\varphi_1\frac{\partial\psi}{\partial y}\phi_3 \;dxdt+\int_0^T\int_{E_r}\frac{\partial\phi_3}{\partial y}\left(\frac{\varphi_3^2}{\varphi_1}+\frac{g}{2}\varphi_1^2\right) \;dxdt.
\end{align*}
So dividing by $s=\alpha_{2-d}r^{2-d}>0$ and applying the Lebesgue differentiation theorem in its more general form (See, Theorem \ref{diffLbg})
\begin{eqnarray*}
d_s\lag(0,\varphi,\Phi)&=&-\int_0^T\int_{E}\vert\nabla \varphi_1\vert^2+\vert\nabla \varphi_2\vert^2+\vert\nabla\varphi_3\vert^2\;dH^ddt\\&\;&-\int_0^T\int_{E}\vert \varphi_1-u_{1d}\vert^2+\vert\varphi_2-u_{2d}\vert^2+\vert \varphi_3-u_{3d}\vert^2\;dH^ddt\\
&\;&+\int_0^T\int_{E}\frac{\partial \phi_1}{\partial t}\varphi_1\;dH^ddt+\int_{E}h_0\phi_1(x,0)\;dH^d\\ 
&\;& +\int_0^T\int_{E}\frac{\partial \phi_1}{\partial x}\varphi_2\;dH^ddt+\int_0^T\int_{E}\frac{\partial \phi_1}{\partial y}\varphi_3\;dH^ddt+\int_0^T\int_{E}\frac{\partial \phi_2}{\partial t}\varphi_2\;dH^ddt\\
&\;&+\int_{E}\qu_{10}\phi_2(x,0)\;dH^d+\int_0^T\int_{E}\frac{\partial\phi_2}{\partial x}\left(\frac{\varphi_2^2}{\varphi_1}+\frac{g}{2}\varphi_1^2\right) \;dH^ddt\\ 
&\;&+\int_0^T\int_{E}\frac{\partial\phi_2}{\partial
	y}\left(\frac{\varphi_2\varphi_3}{\varphi_1}\right)\;dH^ddt-\int_0^T\int_{E}g\varphi_1\frac{\partial\psi}{\partial x}\phi_2 \;dxdt\\	
&\;&+\int_0^T\int_{E}\frac{\partial \phi_3}{\partial t}\varphi_3\;dH^ddt+\int_{E}\qu_{20}\phi_3(x,0)\;dH^d+\int_0^T\int_{E}\frac{\partial\phi_3}{\partial
	x}\left(\frac{\varphi_2\varphi_3}{\varphi_1}\right)\;dH^ddt\\
&\;&-\int_0^T\int_{E}g\varphi_1\frac{\partial\psi}{\partial y}\phi_3 \;dH^ddt+\int_0^T\int_{E}\frac{\partial\phi_3}{\partial y}\left(\frac{\varphi_3^2}{\varphi_1}+\frac{g}{2}\varphi_1^2\right) \;dH^ddt.
\end{eqnarray*}
$H^d$ is the $d$-dimensional Hausdorff measure. We need to calculate the derivative of the perturbed Lagrangian with respect to the variable $\varphi$ in the direction $\varphi'$, $d_\varphi\lag(s,\varphi,\Phi;\varphi')$. We will do this step by step, calculating the derivatives with respect to the variables $ \varphi_1$, $ \varphi_2$ and $ \varphi_3$ in the respective directions $ \varphi_1'$, $ \varphi_2'$ and $ \varphi_3'$
\begin{eqnarray*}
d_{\varphi_1}\lag(s,\varphi,\Phi;\varphi_1')&=&\int_0^T\int_{\Omega_s} 2\nabla\varphi_1\nabla\varphi_1'+ 2(\varphi_1-u_{1d})\varphi_1'\;dxdt+
\\
&\;&\int_0^T\int_{\Omega_s}-\frac{\partial \phi_1}{\partial t}\varphi_1'\;dxdt-\int_0^T\int_{\Omega_s}\frac{\partial\phi_2}{\partial x}\left(-\frac{\varphi_2^2}{\varphi_1^2}+g\varphi_1\right)\varphi_1' \;dxdt+\\ 
&\;&\int_0^T\int_{\Omega_s}-\frac{\partial\phi_2}{\partial
	y}\left(-\frac{\varphi_2\varphi_3}{\varphi_1^2}\right)\varphi_1'\;dxdt+\int_0^T\int_{\Omega_s}g\varphi_1'\frac{\partial\psi}{\partial x}\phi_2 \;dxdt+\\	
&\;& \int_0^T\int_{\Omega_s}-\frac{\partial\phi_3}{\partial
	x}\left(-\frac{\varphi_2\varphi_3}{\varphi_1^2}\right)\varphi_1'\;dxdt+\\
&\;&\int_0^T\int_{\Omega_s}g\varphi_1'\frac{\partial\psi}{\partial y}\phi_3 \;dxdt+\int_0^T\int_{\Omega_s}-\frac{\partial\phi_3}{\partial y}\left(\frac{\varphi_3^2}{\varphi_1^2}+g\varphi_1\right)\varphi_1' \;dxdt.
\end{eqnarray*}
\begin{eqnarray*}
d_{\varphi_2}\lag(s,\varphi,\Phi;\varphi_2')&=&\int_0^T\int_{\Omega_s} 2\nabla\varphi_2\nabla\varphi_2'+ 2(\varphi_2-u_{2d})\varphi_2'\;dxdt+\int_0^T\int_{\Omega_s}-\frac{\partial \phi_1}{\partial x}\varphi_2'\;dxdt+\\	
&\;&\int_0^T\int_{\Omega_s}-\frac{\partial \phi_2}{\partial t}\varphi_2'\;dxdt+\int_0^T\int_{\Omega_s}-\frac{\partial\phi_2}{\partial x}\left(2\frac{\varphi_2}{\varphi_1}\right)\varphi_2' \;dxdt+\\ 
&\;&\int_0^T\int_{\Omega_s}-\frac{\partial\phi_2}{\partial
	y}\left(\frac{\varphi_3}{\varphi_1}\right)\varphi_2'\;dxdt+\int_0^T\int_{\Omega_s}-\frac{\partial\phi_3}{\partial
	x}\left(\frac{\varphi_3}{\varphi_1}\right)\varphi_2'\;dxdt.\\
\end{eqnarray*}

\begin{eqnarray*}
d_{\varphi_3}\lag(s,\varphi,\Phi;\varphi_3')&=&\int_0^T\int_{\Omega_s} 2\nabla\varphi_3\nabla\varphi_3'+ 2(\varphi_3-u_{3d})\varphi_3'\;dxdt
+\int_0^T\int_{\Omega_s}-\frac{\partial \phi_1}{\partial y}\varphi_3'\;dxdt+\\	
&\;&\int_0^T\int_{\Omega_s}-\frac{\partial \phi_3}{\partial t}\varphi_3'\;dxdt+\int_0^T\int_{\Omega_s}-\frac{\partial\phi_2}{\partial y}\frac{\varphi_2}{\varphi_1}\varphi_3' \;dxdt\\ &\;&\int_0^T\int_{\Omega_s}-2\frac{\partial\phi_3}{\partial y}\frac{\varphi_3}{\varphi_1}\varphi_3'\;dxdt-\int_0^T\int_{\Omega_s}\frac{\partial\phi_3}{\partial x}\frac{\varphi_2}{\varphi_1}\varphi_3'\;dxdt.
\end{eqnarray*}
In what follows, we simply write $P$ instead of $P^0$. The latter is nothing more than $P^s=(\pe_1^s,\pe_2^s,\pe_3^s)$ for $s=0$. Now if we wish to obtain the equations of state, we differentiate the perturbed Lagrangian with respect to $\Phi$. In this case, for all $\Phi'$ the state equation at $s=0$ is $d_\Phi\lag(0,U,0;\Phi')=0$.\\
Also we obtain the adjoint equations, by differentiating the perturbed Lagrangian with respect to the variable $\varphi$ and the state adjoint $P$ verifies  the equation $d_\varphi\lag(0,U,P;\varphi')=0$ for all $\varphi'$.   We can state the following theorem:

\begin{theorem}\label{theosytemeadjSWE}
Assume that there exists a solution $U$ of the system (\ref{equaSV}). Then, with the objective function (\ref{shfcgen}), the adjoint in strong form is given by:
\begin{align*}
	&\,	-\frac{\partial \pe_1}{\partial t}+\frac{\partial\pe_2}{\partial x}\left(\frac{\qu_1^2}{h^2}-gh\right)+\frac{\partial\pe_2}{\partial
		y}\frac{\qu_1\qu_2}{h^2}+g\frac{\partial\psi}{\partial x}\pe_2  +\frac{\partial\pe_3}{\partial
		x}\frac{\qu_1\qu_2}{h^2}+g\frac{\partial\psi}{\partial y}\pe_3 \\
	&\,+\frac{\partial\pe_3}{\partial y}\left(\frac{\qu_2^2}{h^2}-gh\right)=2\Delta h- 2(h-u_{1d})
	\\
&\,	-\frac{\partial \pe_2}{\partial t}-\frac{\partial \pe_1}{\partial x}-2\frac{\partial\pe_2}{\partial x}\frac{\qu_1}{h} -\frac{\partial\pe_2}{\partial
	y}\frac{\qu_2}{h}-\frac{\partial\pe_3}{\partial
	x}\frac{\qu_2}{h}=2\Delta\qu_2- 2(\qu_2-u_{2d})
\\
&\, -\frac{\partial \pe_3}{\partial t}	-\frac{\partial \pe_1}{\partial y}-\frac{\partial\pe_2}{\partial y}\frac{\qu_1}{h}-2\frac{\partial\pe_3}{\partial y}\frac{\qu_2}{h}-\frac{\partial\pe_3}{\partial x}\frac{\qu_1}{h}=2\Delta\qu_3-2(\qu_3-u_{3d}).
\end{align*}

with the final conditions
\begin{eqnarray}
P(x,T)=0\;\;\;\text{in}\;\;\Omega
\end{eqnarray}
and  boundary conditions
\begin{align*}
&\;\pe_2n_x+\pe_3n_y=0\;\;\;\text{in}\;\;\partial\Omega\times [0,T]\\&\,
\left(\frac{\qu_1^2}{h}+\frac{g}{2}h^2\right)\pe_2 n_x+\left(\frac{\mathrm{q}_1\qu_2}{h}\right)\pe_2n_y+\left(\frac{\mathrm{q}_1\qu_2}{h}\right)\pe_3n_x+\left(\frac{\qu_2^2}{h}+\frac{g}{2}h^2\right)\pe_3 n_y=0\;\;\;\text{in}\;\;\partial\Omega\times [0,T].
\end{align*}
\end{theorem}
\begin{proof} As mentioned earlier, $P$, the solution of the adjoint equation at $s=0$ verifies $d_{\varphi}\lag(0,U,P;\varphi')=0$ for all $\varphi'=(\varphi_1',\varphi_2',\varphi_3')\in \mathcal{V}_{\Omega}$. Let us first note that
$d_{\varphi}\lag(0,\varphi,\Phi;\varphi')=d_{\varphi_1}\lag(0,\varphi,\Phi;\varphi_1')+d_{\varphi_2}\lag(0,\varphi,\Phi;\varphi_2')+d_{\varphi_3}\lag(0,\varphi,\Phi;\varphi_3')$. So for all $\varphi' \in \mathcal{W}_{\Omega}$
\begin{eqnarray*} 
&\;&\int_0^T\int_{\Omega} 2\nabla h\nabla\varphi_1'+ 2(h-u_{1d})\varphi_1'\;dxdt
+\\ 
&\;&\int_0^T\int_{\Omega}-\frac{\partial \pe_1}{\partial t}\varphi_1'\;dxdt-\int_0^T\int_{\Omega}\frac{\partial\pe_2}{\partial x}\left(-\frac{\qu_1^2}{h^2}+gh\right)\varphi_1' \;dxdt+\\ 
&\;&\int_0^T\int_{\Omega}-\frac{\partial\pe_2}{\partial
	y}\left(-\frac{\qu_1\qu_2}{h^2}\right)\varphi_1'\;dxdt+\int_0^T\int_{\Omega}g\varphi_1'\frac{\partial\psi}{\partial x}\pe_2 \;dxdt+\\	
&\;& \int_0^T\int_{\Omega}-\frac{\partial\pe_3}{\partial
	x}\left(-\frac{\qu_1\qu_2}{h^2}\right)\varphi_1'\;dxdt+\\
&\;&\int_0^T\int_{\Omega}g\varphi_1'\frac{\partial\psi}{\partial y}\pe_3 \;dxdt+\int_0^T\int_{\Omega}-\frac{\partial\pe_3}{\partial y}\left(\frac{\qu_2^2}{h^2}+gh\right)\varphi_1' \;dxdt+\\
\end{eqnarray*}
\begin{eqnarray*}
&\;&\int_0^T\int_{\Omega} 2\nabla\qu_1\nabla\varphi_2'+ 2(\qu_1-u_{2d})\varphi_2'\;dxdt+\int_0^T\int_{\Omega}-\frac{\partial \pe_1}{\partial x}\varphi_2'\;dxdt+\\	
&\;&\int_0^T\int_{\Omega}-\frac{\partial \pe_2}{\partial t}\varphi_2'\;dxdt+\int_0^T\int_{\Omega}-\frac{\partial\pe_2}{\partial x}\left(2\frac{\qu_1}{h}\right)\varphi_2' \;dxdt+\\ 
&\;&\int_0^T\int_{\Omega}-\frac{\partial\pe_2}{\partial
	y}\left(\frac{\qu_2}{h}\right)\varphi_2'\;dxdt+\int_0^T\int_{\Omega}-\frac{\partial\pe_3}{\partial
	x}\left(\frac{\qu_2}{h}\right)\varphi_2'\;dxdt+\\
&\,&\int_0^T\int_{\Omega} 2\nabla\qu_2\nabla\varphi_3'+ 2(\qu_2-u_{3d})\varphi_3'\;dxdt
+\int_0^T\int_{\Omega}-\frac{\partial \pe_1}{\partial y}\varphi_3'\;dxdt+\\	
&\;&\int_0^T\int_{\Omega}-\frac{\partial \pe_1}{\partial t}\varphi_3'\;dxdt+\int_0^T\int_{\Omega}-\frac{\partial\pe_2}{\partial y}\frac{\qu_1}{h}\varphi_3' \;dxdt\\ &\;&\int_0^T\int_{\Omega}-2\frac{\partial\pe_3}{\partial y}\frac{\qu_2}{h}\varphi_3'\;dxdt-\int_0^T\int_{\Omega}\frac{\partial\pe_3}{\partial x}\frac{\qu_1}{h}\varphi_3'\;dxdt=0.
\end{eqnarray*}
For any $ \varphi'$, we have the following equation
\begin{align*}
&\,\left[-2\Delta h+2(h-u_{1d})-\frac{\partial \pe_1}{\partial t}+\frac{\partial\pe_2}{\partial x}\left(\frac{\qu_1^2}{h^2}-gh\right)+\frac{\partial\pe_2}{\partial
	y}\frac{\qu_1\qu_2}{h^2}+g\frac{\partial\psi}{\partial x}\pe_2  \right]\varphi_1'+\\
&\,\left[\frac{\partial\pe_3}{\partial
	x}\frac{\qu_1\qu_2}{h^2}+g\frac{\partial\psi}{\partial y}\pe_3 +\frac{\partial\pe_3}{\partial y}\left(\frac{\qu_2^2}{h^2}-gh\right)\right]\varphi_1'+
\\
&\;\left[-2\Delta\qu_2+2(\qu_2-u_{2d})-\frac{\partial \pe_2}{\partial t}-\frac{\partial \pe_1}{\partial x}-2\frac{\partial\pe_2}{\partial x}\frac{\qu_1}{h} -\frac{\partial\pe_2}{\partial
	y}\frac{\qu_2}{h}-\frac{\partial\pe_3}{\partial
	x}\frac{\qu_2}{h}
\right]\varphi_2'+
\\
&\;	\left[-2\Delta\qu_3+2(\qu_3-u_{3d})-\frac{\partial \pe_1}{\partial t}	-\frac{\partial \pe_1}{\partial y}-\frac{\partial\pe_2}{\partial y}\frac{\qu_1}{h}-2\frac{\partial\pe_3}{\partial y}\frac{\qu_2}{h}-\frac{\partial\pe_3}{\partial x}\frac{\qu_1}{h}\right]\varphi_3'=0.
\end{align*}
This equation can be seen as a system of three equations constituting the system of adjoint equations
\begin{align*}
&\,	-\frac{\partial \pe_1}{\partial t}+\frac{\partial\pe_2}{\partial x}\left(\frac{\qu_1^2}{h^2}-gh\right)+\frac{\partial\pe_2}{\partial
	y}\frac{\qu_1\qu_2}{h^2}+g\frac{\partial\psi}{\partial x}\pe_2  +\frac{\partial\pe_3}{\partial
	x}\frac{\qu_1\qu_2}{h^2}+g\frac{\partial\psi}{\partial y}\pe_3 \\
&\,+\frac{\partial\pe_3}{\partial y}\left(\frac{\qu_2^2}{h^2}-gh\right)=2\Delta h- 2(h-u_{1d})
\\
&\,	-\frac{\partial \pe_2}{\partial t}-\frac{\partial \pe_1}{\partial x}-2\frac{\partial\pe_2}{\partial x}\frac{\qu_1}{h} -\frac{\partial\pe_2}{\partial
	y}\frac{\qu_2}{h}-\frac{\partial\pe_3}{\partial
	x}\frac{\qu_2}{h}=2\Delta\qu_2- 2(\qu_2-u_{2d})
\\
&\, -\frac{\partial \pe_3}{\partial t}	-\frac{\partial \pe_1}{\partial y}-\frac{\partial\pe_2}{\partial y}\frac{\qu_1}{h}-2\frac{\partial\pe_3}{\partial y}\frac{\qu_2}{h}-\frac{\partial\pe_3}{\partial x}\frac{\qu_1}{h}=2\Delta\qu_3-2(\qu_3-u_{3d}).
\end{align*}
The terms on the right-hand side are derived from an integration by parts.  We apply Green's formula taking into account the boundary conditions on $\Omega$. This system is equivalent to
\begin{eqnarray}\label{Matriceadjoint}
-\frac{\partial P}{\partial t}+AP_x+BP_y+CP+S=0
\end{eqnarray}
which is nothing but the resulting system of adjoint equations written in matrix form with $P=(\pe_1,\pe_2,\pe_3)$ and the matrices $A$, $B$ and $C$ are given by 
\begin{eqnarray*}
A= \begin{pmatrix}
	0&\frac{\qu_1^2}{h^2}-gh&\frac{\qu_1\qu_2}{h^2}\\-1&-\frac{2\qu_1}{h}&-\frac{\qu_2}{h}\\0&0&-\frac{\qu_1}{h}
\end{pmatrix},\;\;\;
B=\begin{pmatrix}
	0&\frac{\qu_1\qu_2}{h^2}&\frac{\qu_2^2}{h^2}-gh\\ 0&-\frac{\qu_2}{h}&0\\-1&-\frac{q_1}{h}&-\frac{2\qu_2}{h}
\end{pmatrix}\;\;\text{and}\;\;\; 	C=\begin{pmatrix}
	0&g\frac{\partial\psi}{\partial x}&g\frac{\partial\psi}{\partial y}\\ 0&0&0\\ 0&0&0
\end{pmatrix}.
\end{eqnarray*}
The vector $S$ is given by $S=\Big(2\Delta h- 2(h-u_{1d}),2\Delta\qu_2- 2(\qu_2-u_{2d}),2\Delta\qu_3-2(\qu_3-u_{3d})\Big)^T$.
\end{proof}
\subsection{A necessary condition for existence of the topological derivative}
The $R(U,P)$ function is defined as the limit of the $R(s)$ function when $s$ goes to zero. $R(s)$ is defined by
\begin{eqnarray*}
R(U,P)=\lim_{s\to 0^+}\, R(s)&=&\lim_{s\to 0^+}\int_0^1 d_{\varphi}\lag\left(s,U+\theta(U^s-U), P,\frac{U^s-U}{s}\right)\;d\theta\\
&=&\lim_{s\to 0^+}\frac{1}{s}\int_0^1 d_{\varphi}\lag\left(s,U+\theta(U^s-U), P,U^s-U\right)\;d\theta.
\end{eqnarray*}
The function of the $s$ parameter, $R(s)$ is given explicitly by the following expression
\begin{equation*}
\begin{aligned}
R(s)&=\frac{1}{s}\int_0^1\int_0^T\int_{\Omega}\Big[ 2\nabla(h+\theta(h^s-h))\nabla(h^s-h)+ 2(h+\theta(h^s-h)-u_{1d})(h^s-h)
\\[0.2cm] &-\frac{\partial \pe_1}{\partial t}(h^s-h)-\frac{\partial\pe_2}{\partial x}\left(-\frac{(\qu_1+\theta(\qu_1^s-\qu_1))^2}{(h+\theta(h^s-h))^2}+g(h+\theta(h^s-h))\right)(h^s-h) \\[0.2cm] 
&-\frac{\partial\pe_2}{\partial
	y}\left(-\frac{(\qu_1+\theta(\qu_1^s-\qu_1))(\qu_2+\theta(\qu_2^s-\qu_2))}{(h+\theta(h^s-h))^2}\right)(h^s-h)+g(h^s-h)\frac{\partial\psi}{\partial x}\pe_2 \\[0.2cm]	
& -\frac{\partial\pe_3}{\partial
	x}\left(-\frac{(\qu_1+\theta(\qu_1^s-\qu_1))(\qu_2+\theta(\qu_2^s-\qu_2))}{(h+\theta(h^s-h))^2}\right)(h^s-h)+g(h^s-h)\frac{\partial\psi}{\partial y}\pe_3 \\[0.2cm] 
& -\frac{\partial\pe_3}{\partial y}\left(\frac{(\qu_2+\theta(\qu_2^s-\qu_2))^2}{(h+\theta(h^s-h))^2}+g(h+\theta(h^s-h))\right)(h^s-h) \\[0.2cm]
&+ 2\nabla(\qu_1+\theta(\qu_1^s-\qu_1))\nabla(\qu_1^s-\qu_1)+  2(\qu_1+\theta(\qu_1^s-\qu_1-u_{2d}))(\qu_1^s-\qu_1)\\
& -\frac{\partial \pe_1}{\partial x}(\qu_1^s-\qu_1)-\frac{\partial \pe_2}{\partial t}(\qu_1^s-\qu_1) -\frac{\partial\pe_2}{\partial x}\left(2\frac{\qu_1+\theta(\qu_1^s-\qu_1)}{h+\theta(h^s-h)}\right)(\qu_1^s-\qu_1) \\[0.2cm]
\end{aligned}
\end{equation*}
\begin{equation*}
\begin{aligned}
&-\frac{\partial\pe_2}{\partial
	y}\left(\frac{\qu_2+\theta(\qu_2^s-\qu_2)}{h+\theta(h^s-h)}\right)(\qu_1^s-\qu_1) -\frac{\partial\pe_3}{\partial
	x}\left(\frac{\qu_2+\theta(\qu_2^s-\qu_2)}{h+\theta(h^s-h)}\right)(\qu_1^s-\qu_1)\\[0.2cm]
& + 2\nabla(\qu_2+\theta(\qu_2^s-\qu_2)\nabla(\qu_2^s-\qu_2)+  2(\qu_2+\theta(\qu_2^s-\qu_2)-u_{3d})(\qu_2^s-\qu_2)
\\[0.2cm] &-\frac{\partial \pe_1}{\partial y}(\qu_2^s-\qu_2)-\frac{\partial \pe_1}{\partial t}(\qu_2^s-\qu_2)-\frac{\partial\pe_2}{\partial y}\frac{\qu_1+\theta(\qu_1^s-\qu_1)}{h+\theta(h^s-h)}(\qu_2^s-\qu_2) \\[0.2cm]&-2\frac{\partial\pe_3}{\partial y}\frac{\qu_2+\theta(\qu_2^s-\qu_2)}{h+\theta(h^s-h)}(\qu_2^s-\qu_2)-\frac{\partial\pe_3}{\partial x}\frac{\qu_1+\theta(\qu_1^s-\qu_1)}{h+\theta(h^s-h)}(\qu_2^s-\qu_2)\Big] \;dxdtd\theta.
\end{aligned}
\end{equation*}				
The existence of the topological derivative depends strongly on the limit of the function $R(s)$ when $s$ goes to zero. Since we are dealing here with a non-linear PDE, it is not easy to compute the limit of $R(s)$. However, if the limit exists, we can state the following theorem, which we have just demonstrated.
\begin{theorem}
Assume that the perturbed state equation (\ref{equaSV}) and the adjoint state system given in the Theorem \ref{theosytemeadjSWE} admit each a unique solution. Then, under hypothesis (H0), if the function $R(s)$ admits a finite limit denoted $R(U,P)$, then the topological derivative of the objective function exists and is given by
\begin{eqnarray*}
DJ_T&=&R(U,P)-\int_0^T\int_{E}\vert\nabla h\vert^2+\vert\nabla \qu_1\vert^2+\vert\nabla\qu_2\vert^2\;dH^ddt\\&\;&-\int_0^T\int_{E}\vert h-u_{1d}\vert^2+\vert\qu_1-u_{2d}\vert^2+\vert \qu_2-u_{3d}\vert^2\;dH^ddt\\
&\;&+\int_0^T\int_{E}\frac{\partial \pe_1}{\partial t}h\;dH^ddt+\int_{E}h_0\pe_1(x,0)\;dH^d\\ 
&\;& +\int_0^T\int_{E}\frac{\partial \pe_1}{\partial x}\qu_1\;dH^ddt+\int_0^T\int_{E}\frac{\partial \pe_1}{\partial y}\qu_2\;dH^ddt+\int_0^T\int_{E}\frac{\partial \pe_2}{\partial t}\qu_1\;dH^ddt\\
&\;&+\int_{E}\qu_{10}\pe_2(x,0)\;dH^d+\int_0^T\int_{E}\frac{\partial\pe_2}{\partial x}\left(\frac{\qu_1^2}{h}+\frac{g}{2}h^2\right) \;dH^ddt\\ 
&\;&+\int_0^T\int_{E}\frac{\partial\pe_2}{\partial
	y}\left(\frac{\qu_1\qu_2}{h}\right)\;dH^ddt-\int_0^T\int_{E}gh\frac{\partial\psi}{\partial x}\pe_2 \;dxdt+\int_0^T\int_{E}\frac{\partial \pe_3}{\partial t}\pe_2\;dH^ddt\\	
&\;&+\int_{E}\qu_{20}\pe_3(x,0)\;dH^d+\int_0^T\int_{E}\frac{\partial\pe_3}{\partial x}\left(\frac{\qu_1\qu_2}{h}\right)\;dH^ddt\\
&\;&-\int_0^T\int_{E}gh\frac{\partial\psi}{\partial y}\pe_3 \;dH^ddt+\int_0^T\int_{E}\frac{\partial\pe_3}{\partial y}\left(\frac{\qu_2^2}{h}+\frac{g}{2}h^2\right) \;dH^ddt.
\end{eqnarray*}
where $U=(h,\qu_1,\qu_2)$ and $P=(\pe_1,\pe_2,\pe_3)$ are respectively solutions of (\ref{equaSV}) and (\ref{Matriceadjoint}).
\end{theorem}

\section{Topological derivative using the Gangl and Sturm method}\label{secttopogangl}
\noindent
With the min-max approach developed in the previous sections, it is not at all easy to compute the limit of the function $R(s)$ when goes to zero.  We now turn to the ideas recently developed by  Gangl and Sturm \cite{ganglautomated}, which allow to avoid the explicit computation of the limit $R(U_0,P_0)$. But in this method, the main challenges are to fulfill to main hypotheses. One of them need to  be able to claim the existence of a corrector term. To this end, we propose an alternative writing of the system (\ref{equaSV}) where we add an extra term as follows:
\begin{eqnarray}\label{swevisqueux}
\frac{\partial U}{\partial t}+\nabla\cdot F(U)-\nabla\cdot(Q(\alpha)\nabla U)=S(U)\qquad\qquad\text{on}\quad \Omega\times[0,T]
\end{eqnarray}
where
\begin{eqnarray*}
F(U)=\begin{pmatrix}
hu&hv\\ hu^2+\frac{1}{2}gh^2&huv \\  huv & hv^2+\frac{1}{2}gh^2
\end{pmatrix}\qquad\text{and}\qquad
S(U)=\begin{pmatrix}
0\\ -gh\frac{\partial\psi}{\partial x}\\ -gh\frac{\partial\psi}{\partial y}
\end{pmatrix}.
\end{eqnarray*}
and the same initial and boundary conditions (\ref{ICSWE}) and (\ref{BCSWE}) apply here. In this new system, we have also introduced an additional term $\nabla\cdot(Q(\alpha)\nabla U) $, the Laplacian artificial viscosity. Here, the $Q(\alpha)$ coefficient is a parameter that controls the amount of viscosity. For further details on viscosity coefficients, please refer to \cite{olofsub}.
$Q(\alpha)$ is a diagonal matrix of size $3\times 3$ defined by
\begin{eqnarray*}
Q(\alpha)=\begin{pmatrix}
0&0&0\\ 0&\alpha_1 &0 \\0 &0&\alpha_2
\end{pmatrix}
\end{eqnarray*}
with  $\alpha_1, \alpha_2\in \mathbb{R}_+$ such that 
\begin{eqnarray*}
\nabla\cdot(Q(\alpha)\nabla U)=\begin{pmatrix}
0\\ \alpha_1\nabla^2 (hu)\\ \alpha_2\nabla^2 (hv)
\end{pmatrix}.
\end{eqnarray*}
Here we have chosen to take zero as the first element of the diagonal of the matrix $Q(\alpha)$. This choice is justified by the fact that, in general, no viscosity is placed on the continuity equation in the viscous shallow water system, making it a non-fully parabolic system. One of the simple reasons why it may be necessary to place a viscosity on the continuity equation is to avoid the discontinuities that can occur in the original formulation of the hyperbolic shallow water equation \cite{lukaswe}. \\In what follows, our aim is to compute the topological derivative of the objective function $J(\Omega)$, defined by 
\begin{eqnarray}
J(\Omega)=\int_{0}^T\int_{\Omega}j^\Omega(x,U,\nabla U)\;dxdt.
\end{eqnarray}
where effectively $U$ is a solution of the viscous shallow water equation (\ref{swevisqueux}). For this purpose, we consider a computational domain $D$ which is subdivided into two open disjoint subdomains: $\Omega$ and its complementary in $D$, i.e. $D=\Omega^{int}\cup\Omega^{ext}$ where $\Omega^{int}=\Omega$ and $\Omega^{ext}=D\backslash\overline{\Omega}$. We can then write
\begin{eqnarray}
J(\Omega)=\int_{0}^T\int_{D}j^\Omega(x,U,\nabla U)\;dxdt
\end{eqnarray}
with $j^\Omega(x,U,\nabla U)$ is equal to zero outside of $\Omega$ or simply
\begin{eqnarray}\label{JsurD}
j^\Omega(x, U,\nabla U)=\chi_\Omega(x)j^{int}(x, U,\nabla U)+\chi_{D\backslash\Omega}(x)j^{ext}(x, U,\nabla U),\quad \text{with}\;\; j^{ext}(x, U,\nabla U)=0
\end{eqnarray}	
and
\begin{eqnarray*}
j^{int}(x, U,\nabla U)=\vert \nabla U\vert^2+\vert U-U_d\vert^2
\end{eqnarray*}
where $U\in H^1(D\times [0,T])^3$ is a solution to the variational equation
\begin{align*}
&\;\int_0^T\int_{D}\partial_t(U)\cdot V\;dxdt+\int_0^T\int_{D}(\nabla\cdot(F(U)))\cdot V\;dxdt+\int_0^T\int_{D}Q(\alpha)\nabla U:\nabla V\;dxdt-\\&\;\int_0^T\int_{D}S(U)\cdot V\;dxdt=0,\;\;\forall\; V\in H^1_0(D\times [0,T])^3.
\end{align*}
This last equation can be rewritten as follows
\begin{align}\label{swe1}
&\;\int_0^T\int_{D}\left( \partial_t (U)+\nabla\cdot( F(U))- S(U)\right)\cdot V\;dxdt+\int_0^T\int_{D} Q(\alpha)\nabla U:\nabla V\;dxdt=0.
\end{align}
In this case the PDE-constrained topology optimization problem is written:
\begin{eqnarray*}
	&\;&\min_{\Omega\in\mathcal{O}_{ad}} \quad	J(\Omega)=\int_{0}^T\int_{D}j^\Omega(x,U,\nabla U)\;dxdt\\[0.2cm]
	&\,&\quad s.t. \quad U \;\text{solves}\; (\ref{swe1})
\end{eqnarray*} 
where $j^\Omega$ is defined in (\ref{JsurD}) and $\mathcal{O}_{ad}$ is the set of admissible subsets of $D$. We consider the following operators
\begin{eqnarray*}
&\;&A_1^{int}, A_1^{ext}:D\times\mathbb{R}^3\times\mathbb{R}^{3\times 2}\rightarrow \mathbb{R}^{3}\\
&\;&A_2^{int}, A_2^{ext}:D\times\mathbb{R}^3\times\mathbb{R}^{3\times 2}\rightarrow \mathbb{R}^{3\times 2}
\end{eqnarray*}	
and we define the two piecewise operators $A_1^\Omega:D\times\mathbb{R}^3\times\mathbb{R}^{3\times 2}\rightarrow \mathbb{R}^{3}$ and $A_2^\Omega:D\times\mathbb{R}^3\times\mathbb{R}^{3\times 2}\rightarrow \mathbb{R}^{3\times 2}$
\begin{eqnarray*}
A_1^\Omega(x,y_1,y_2)=\chi_\Omega(x)A_1^{int}(x,y_1,y_2)+\chi_{D\backslash\Omega}(x)A_1^{ext}(x,y_1,y_2)\\
A_2^\Omega(x,y_1,y_2)=\chi_\Omega(x)A_2^{int}(x,y_1,y_2)+\chi_{D\backslash\Omega}(x)A_2^{ext}(x,y_1,y_2).
\end{eqnarray*}
We set
\begin{eqnarray*}
&\,&A_1^\Omega(x,U,\nabla U)=\partial_t(U)+\nabla\cdot(F(U))- S(U)\in \mathbb{R}^{3}\\&\;& A_2^\Omega(x,U,\nabla U)=Q(\alpha)\nabla U\in\mathbb{R}^{3\times 2}.
\end{eqnarray*}
For a function $\varphi:\mathbb{R}^2\rightarrow\mathbb{R}^3$, we denote by $\nabla\varphi\in\mathbb{R}^{3\times 2}$ its Jacobian, $(\nabla\varphi)_{i,k}=\frac{\partial\varphi_i}{\partial x_k}$ for $i=\{1,2,3\}$, $k\in\{1,2\}$. The Euclidean vector product $a\cdot b=\sum_{1}^{3}a_ib_i$ for $a,b\in\mathbb{R}^3$. $A:B=\sum_{i=1}^{3}\sum_{k=1}^{2}A_{i,k}B_{i,k}$ is the Frobenius inner product of two matrices $A,B\in\mathbb{R}^{3\times 2}$. We would like to underline that in all that follows, we use the notation $\partial_U\mathcal{L}(\epsilon,U_\epsilon,0)(\varphi)$ to designate the partial derivative of the Lagrangian with respect to the variable $U$ in the $\varphi$-direction. In the result we want to demonstrate, we need to remember the following definitions and theorem which can be found in \cite{ganlsimplified}\cite{ganglautomated}.
\begin{definition} \cite{ganlsimplified}\label{defstatandadjoint}\\
Parametrized Lagrangian: Let $X$ and $Y$ be two vector spaces and $\tau>0$. A parameterized Lagrangian (or short Lagrangian) is a function 
\begin{eqnarray*}
(\epsilon,U,P)\mapsto \mathcal{L}(\epsilon,U,P):[0,\tau]\times X\times Y\rightarrow\mathbb{R}
\end{eqnarray*}
satisfying 
\begin{eqnarray*}
P\mapsto \mathcal{L}(\epsilon,U,P)\;\text{is affine on $Y$}.
\end{eqnarray*}
State and adjoint state: Let $\epsilon\in [0,\tau]$. We define the state equation by: find $U_\epsilon\in X$ such that
\begin{eqnarray*}
\partial_P\mathcal{L}(\epsilon,U_\epsilon,0)(\varphi)=0,\;\;\text{for all $\varphi\in Y$.}
\end{eqnarray*}
and the set of states is denoted $E(\epsilon)$. We also define the adjoint state equation by: find $P_\epsilon\in Y$, such that
\begin{eqnarray*}
\partial_U \lag(\epsilon,U_\epsilon,P_\epsilon)(\varphi)=0,\;\;\text{for all $\varphi\in X$.}
\end{eqnarray*}
The set of adjoint states associated with $(\epsilon,U_\epsilon)$ is denoted by $Y(\epsilon,U_\epsilon)$.
\end{definition}				
\begin{definition}($\ell$-differentiable Lagrangian \cite{ganlsimplified} )\\
	Let $X$ and $Y$ be two vector spaces and $\tau>0$. Let $\ell:[0,\tau]\rightarrow\mathbb{R}$ a given function satisfying $\ell(0)=0$ and $\ell(\epsilon)>0$ for all $\epsilon\in ]0,\tau]$. A differentiable parametrized Lagrangian is a parametrized Lagrangian. $\lag:[0,\tau]\times X\times Y\rightarrow\mathbb{R}$ satisfying
	\begin{itemize}
		\item[(i)] for all $U,V\in X$ and $P\in Y$
		\begin{eqnarray*}
			\theta\mapsto \lag(\epsilon,U+\theta V,P)\;\;\text{is absolutely continuous on $[0,1]$}
		\end{eqnarray*}
		\item [(ii)] for all $U_0\in E(0)$ and $P_0\in Y(0,U_0)$ the limit
		\begin{eqnarray*}
			\partial_\ell \lag(0,U_0,P_0):=\lim_{\epsilon\rightarrow 0^+}\;\frac{\lag(\epsilon,U_0,P_0)-\lag(0,U_0,P_0)}{\ell(\epsilon)}\;\text{exists.}
		\end{eqnarray*}
	\end{itemize}
\end{definition}
\noindent	
\textbf{Hypothesis (H0a)}\\
\begin{itemize}
	\item[(i)] We assume that for all $\epsilon\in [0,\tau]$, the set $E(\epsilon)=\{U_\epsilon\}$ is a singleton\\
	\item[(ii)] We assume that the adjoint equation for $\epsilon=0$, $\partial_U \mathcal{L}(0,U_0,P_0)(\varphi)$ for all $\varphi\in E$, admits a unique solution.
\end{itemize}
We state the following theorem, which is simply a refinement of Theorem \ref{deltheosalwater}. Its specificity is that it has two $R$ terms instead of one. 
\begin{theorem}\label{theoderivegang}
	Let $\mathcal{L}:[0,\tau]\times X\times Y\rightarrow\mathbb{R}$ a differentiable parametrized Lagrangian satisfying the Hypothesis (H0a). We define for $\epsilon>0$
	\begin{eqnarray}
		R_1^\epsilon(U_0,P_0):=\frac{1}{\ell(\epsilon)}\int_0^1\left[ \partial_U \mathcal{L}(\epsilon,\theta U_\epsilon+(1-\theta)U_0,P_0)-\partial_U \mathcal{L}(\epsilon,U_0,P_0) \right](U_\epsilon-U_0)\;d\theta
	\end{eqnarray}
	\begin{eqnarray}
		R_2^\epsilon(U_0,P_0):=\frac{1}{\ell(\epsilon)}\left[ \partial_U \mathcal{L}(\epsilon,U_0,P_0)-\partial_U \mathcal{L}(0,U_0,P_0) \right](U_\epsilon-U_0).
	\end{eqnarray}
	If $R_1(U_0,P_0)=\lim_{\epsilon\rightarrow 0^+}\; R_1^\epsilon(U_0,P_0)$ and $R_2(U_0,P_0)=\lim_{\epsilon\rightarrow 0^+}\; R_2^\epsilon(U_0,P_0)$ exist, then
	\begin{eqnarray*}
		d_\ell \mathrm{g}(0)=\partial_\ell \mathcal{L}(0,U_0,P_0)+R_1(U_0,P_0)+R_2(U_0,P_0).
	\end{eqnarray*}
\end{theorem}
\begin{proof} See \cite{ganglautomated}.
\end{proof}

\noindent
Let $x_0\in D\backslash\partial\Omega$, in what follows, we no longer consider perturbations with $r$-dilatations, but rather with domains $\omega_\epsilon$ defined by $\omega_\epsilon=x_0+\epsilon\omega$ where $\omega$ is a fixed domain in $\mathbb{R}^2$. In this case, we define the perturbed domain by
\begin{eqnarray*}
	\Omega_\epsilon(x_0)=
	\begin{cases}
		\Omega\backslash\overline{\omega}_\epsilon(x_0)\;\;\text{if}\;\;x_0\in\Omega\\[0.2cm]
		\Omega\cup\omega_\epsilon(x_0)\;\;\text{if}\;\;x_0\in D\backslash\overline{\Omega}.
	\end{cases}
\end{eqnarray*}	
Without loss of generality, we deal with the case where $x_0\in D\backslash\overline{\Omega}$ or in other words $\Omega_\epsilon(x_0)=\Omega\backslash\overline{\omega}_\epsilon(x_0)$ and we can rewrite (\ref{swe1}) as follows
\begin{eqnarray*}
	\int_{0}^T\int_{D}A_1^\Omega(x,U,\nabla U)\cdot\psi \;dxdt+\int_{0}^T\int_{D}  A_2^\Omega(x,U,\nabla U):\nabla\psi \;dxdt=0
\end{eqnarray*}
Let $\psi\in E(\epsilon)$, the perturbed state equation	
\begin{align}\label{perber}
	\int_{0}^T\int_{D} A_1^{(\epsilon)}(x,U_\epsilon,\nabla U_\epsilon)\cdot\psi \;dxdt+\int_{0}^T\int_{D} A_2^{(\epsilon)}(x,U_\epsilon,\nabla U_\epsilon):\nabla\psi\; dxdt=0.
\end{align}	
By subtracting the equation (\ref{perber}) from $\epsilon>0$ and $\epsilon=0$ we have	
\begin{align*}
	&\;\int_{0}^T\int_{D}\left( A_1^{(\epsilon)}(x,U_\epsilon,\nabla U_\epsilon)-  A_1^{(0)}(x,U_0,\nabla U_0)\right)\cdot\psi \;dxdt\\&\;+\int_{0}^T\int_{D}\left(A_2^{(\epsilon)}(x,U_\epsilon,\nabla U_\epsilon)-A_2^{(0)}(x,U_0,\nabla U_0)\right):\nabla\psi \;dxdt=0.
\end{align*}			
We use the change of variables $T_\epsilon(x)=x_0+\epsilon x$ and we note by $D_\epsilon=T_\epsilon^{-1}(D)$	
\begin{align*}
	&\,\int_{0}^T\int_{D_\epsilon}\left( A_1^{(\epsilon)}(T_\epsilon(x),U_\epsilon\circ T_\epsilon,(\nabla U_\epsilon)\circ T_\epsilon)-  A_1^{(0)}(T_\epsilon(x),U_0\circ T_\epsilon,(\nabla U_0)\circ T_\epsilon)\right)\cdot(\psi\circ T_\epsilon) \;dxdt+\\&\;\int_{0}^T\int_{D_\epsilon}\left(A_2^{(\epsilon)}(T_\epsilon(x), U_\epsilon\circ T_\epsilon,(\nabla U_\epsilon)\circ T_\epsilon)-A_2^{(0)}(T_\epsilon(x),U_\epsilon\circ T_\epsilon,(\nabla U_\epsilon)\circ T_\epsilon)\right):((\nabla\psi)\circ T_\epsilon) \;dxdt=0.
\end{align*}			
Using the equality $(\nabla\varphi)\circ T_\epsilon=\frac{1}{\epsilon}\nabla(\varphi\circ T_\epsilon)$ and the first variation of the state $K_\epsilon=\frac{(U_\epsilon-U_0)\circ T_\epsilon}{\epsilon}$, we have
\begin{equation*}
	\begin{aligned}
		\int_{0}^T\int_{D_\epsilon}&\Big[ A_1^{(\epsilon)}(T_\epsilon(x),U_0\circ T_\epsilon+\epsilon K_\epsilon,(\nabla U_0)\circ T_\epsilon+\nabla K_\epsilon)\\&\qquad\qquad\qquad-A_1^{(0)}(T_\epsilon(x),U_0\circ T_\epsilon,(\nabla U_0)\circ T_\epsilon)\Big]\cdot(\psi\circ T_\epsilon) \;dxdt+\\ \int_{0}^T\int_{D_\epsilon}&\frac{1}{\epsilon}\Big[ A_2^{(\epsilon)}(T_\epsilon(x),U_0\circ T_\epsilon+\epsilon K_\epsilon,(\nabla U_0)\circ T_\epsilon+\nabla K_\epsilon)\\&\qquad\qquad\qquad-A_2^{(0)}(T_\epsilon(x),U_0\circ T_\epsilon,(\nabla U_0)\circ T_\epsilon)\Big]:\nabla(\psi\circ T_\epsilon) \;dxdt=0.
	\end{aligned}	
\end{equation*}
Multiplying by $\epsilon$, we get
\begin{equation*}
	\begin{aligned}
		\int_{0}^T\int_{D_\epsilon}&\Big[ A_1^{(\epsilon)}(T_\epsilon(x),U_0\circ T_\epsilon+\epsilon K_\epsilon,(\nabla U_0)\circ T_\epsilon+\nabla K_\epsilon)\\&\qquad\qquad\qquad -A_1^{(0)}(T_\epsilon(x),U_0\circ T_\epsilon,(\nabla U_0)\circ T_\epsilon)\Big]\cdot(\epsilon\psi\circ T_\epsilon) \;dxdt+\\ \int_{0}^T\int_{D_\epsilon}&\Big[ A_2^{(\epsilon)}(T_\epsilon(x),U_0\circ T_\epsilon+\epsilon K_\epsilon,(\nabla U_0)\circ T_\epsilon+\nabla K_\epsilon)\\&\qquad\qquad\qquad-A_2^{(0)}(T_\epsilon(x),U_0\circ T_\epsilon,(\nabla U_0)\circ T_\epsilon)\Big]:\nabla(\psi\circ T_\epsilon) \;dxdt=0.
	\end{aligned}	
\end{equation*}
\textbf{Hypothesis (Hb)}  
\begin{itemize}
	\item[(i)] We assume that $U_0$ is continuously differentiable in $x_0$.
	\item[(ii)] We assume that $\nabla K_\epsilon\rightarrow \nabla K$ and $\epsilon K_\epsilon\rightarrow 0$ where $K$ solves the equation (\ref{equationdeK}).
	\item[(iii)] For all $\epsilon> 0$ we have $\psi\in H^1(D_\epsilon\times [0,T])^3$ if and only if $\psi\circ T_\epsilon\in H^1(D\times [0,T])^3$.
\end{itemize}
Replacing $\psi\circ T_\epsilon$ by $\psi$ the previous equation becomes
\begin{small}
	\begin{align*}
		&\,\int_{0}^T\int_{D_\epsilon}\left[ A_1^{(\epsilon)}(T_\epsilon(x),U_0\circ T_\epsilon+\epsilon K_\epsilon,(\nabla U_0)\circ T_\epsilon+\nabla K_\epsilon)- A_1^{(0)}(T_\epsilon(x),U_0\circ T_\epsilon,(\nabla U_0)\circ T_\epsilon)\right]\cdot(\epsilon\psi)  \;dxdt+\\&\,\int_{0}^T\int_{D_\epsilon}\left[ A_2^{(\epsilon)}(T_\epsilon(x),U_0\circ T_\epsilon+\epsilon K_\epsilon,(\nabla U_0)\circ T_\epsilon+\nabla K_\epsilon)-A_2^{(0)}(T_\epsilon(x),U_0\circ T_\epsilon,(\nabla U_0)\circ T_\epsilon)\right]:\nabla\psi  \;dxdt=0.
	\end{align*}	
\end{small}
or again
\begin{small}
	\begin{align*}
		&\,\int_{0}^T\int_{D_\epsilon}\left[ A_1^{(\epsilon)}(T_\epsilon(x),U_0\circ T_\epsilon+\epsilon K_\epsilon,(\nabla U_0)\circ T_\epsilon+\nabla K_\epsilon)- A_1^{(\epsilon)}(T_\epsilon(x),U_0\circ T_\epsilon,(\nabla U_0)\circ T_\epsilon)\right]\cdot(\epsilon\psi)  \;dxdt+\\&\,\int_{0}^T\int_{D_\epsilon}\left[ A_2^{(\epsilon)}(T_\epsilon(x),U_0\circ T_\epsilon+\epsilon K_\epsilon,(\nabla U_0)\circ T_\epsilon+\nabla K_\epsilon)-A_2^{(\epsilon)}(T_\epsilon(x),U_0\circ T_\epsilon,(\nabla U_0)\circ T_\epsilon)\right]:\nabla\psi  \;dxdt\\&\,\qquad\quad=\int_{0}^T\int_{D_\epsilon}\left( A_1^{(0)}- A_1^{(\epsilon)}\right)\left[(T_\epsilon(x),U_0\circ T_\epsilon,(\nabla U_0)\circ T_\epsilon)\right]\cdot(\epsilon\psi)  \;dxdt
		\\
		&\,\qquad\quad+\int_{0}^T\int_{D_\epsilon}\left(A_2^{(0)}-A_2^{(\epsilon)}\right)(T_\epsilon(x),U_0\circ T_\epsilon,(\nabla U_0)\circ T_\epsilon):\nabla\psi  \;dxdt.
	\end{align*}	
\end{small}
Now we can write
\begin{align*}
	&\,\int_{0}^T\int_{D_\epsilon}\left( A_1^{(0)}- A_1^{(\epsilon)}\right)\left[(T_\epsilon(x),U_0\circ T_\epsilon,(\nabla U_0)\circ T_\epsilon)\right]\cdot(\epsilon\psi)  \;dxdt\\
	&\,\qquad=\int_{0}^T\int_{\omega}\left( A_1^{int}-A_1^{ext} \right)(T_\epsilon(x),U_0\circ T_\epsilon,(\nabla U_0)\circ T_\epsilon)\cdot(\epsilon\psi)  \;dxdt.
\end{align*}
In the same way we show that
\begin{align*}
	&\;\int_{0}^T\int_{D_\epsilon}\left(A_2^{(0)}-A_2^{(\epsilon)}\right)\left[(T_\epsilon(x),U_0\circ T_\epsilon,(\nabla U_0)\circ T_\epsilon)\right]:\nabla\psi  \;dxdt\\&\,\quad\qquad=
	\int_{0}^T\int_{\omega}\left(A_2^{int}-A_2^{ext} \right)\left(T_\epsilon(x),U_0\circ T_\epsilon,(\nabla U_0)\circ T_\epsilon\right):\nabla\psi  \;dxdt,\quad\text{for all $H^1(D_\epsilon\times [0,T])^3$}.
\end{align*}
Finally we get the following equality
\begin{align*}
	&\, \int_{0}^T\int_{D_\epsilon}\left[ A_1^{(\epsilon)}(T_\epsilon(x),U_0\circ T_\epsilon+\epsilon K_\epsilon,(\nabla U_0)\circ T_\epsilon+\nabla K_\epsilon)- A_1^{(\epsilon)}(T_\epsilon(x),U_0\circ T_\epsilon,(\nabla U_0)\circ T_\epsilon)\right]\cdot(\epsilon\psi)  \;dxdt+\\&\, \int_{0}^T\int_{D_\epsilon}\left[ A_2^{(\epsilon)}(T_\epsilon(x),U_0\circ T_\epsilon+\epsilon K_\epsilon,(\nabla U_0)\circ T_\epsilon+\nabla K_\epsilon)-A_2^{(\epsilon)}(T_\epsilon(x),U_0\circ T_\epsilon,(\nabla U_0)\circ T_\epsilon)\right]:\nabla\psi  \;dxdt\\&\;\quad =\int_{0}^T\int_{\omega}\left( A_1^{int}- A_1^{ext} \right)(T_\epsilon(x),U_0\circ T_\epsilon,(\nabla U_0)\circ T_\epsilon)\cdot(\epsilon\psi)  \;dxdt\\
	&\,\qquad  +\int_{0}^T\int_{\omega}\left(A_2^{int}-A_2^{ext} \right)\left(T_\epsilon(x),U_0\circ T_\epsilon,(\nabla U_0)\circ T_\epsilon\right):\nabla\psi  \;dxdt.
\end{align*}	
Going to the limit when $\epsilon$ goes to zero, we get
\begin{align*}
	&\,\int_{0}^T\int_{\mathbb{R}^2}\left[ A_2^{\omega}(x_0,U_0(x_0),\nabla U_0(x_0)+\nabla K)-A_2^{\omega}(x_0,U_0(x_0),\nabla U_0(x_0))\right]:\nabla\psi \;dxdt\\&\,\quad= \int_{0}^T\int_{\omega}\left(A_2^{int}-A_2^{ext} \right)(x_0,U_0(x_0),\nabla U_0(x_0)):\nabla\psi \;dxdt.
\end{align*}
Since $A_2^{ext}$ is zero, we have
\begin{align*}
	\nonumber	&\,\int_{0}^T\int_{\omega}\left[ A_2^{int}(x_0,U_0(x_0),\nabla U_0(x_0)+\nabla K)-A_2^{int}(x_0,U_0(x_0),\nabla U_0(x_0))\right]:\nabla\psi  \;dxdt\\&\,\qquad\quad= \int_{0}^T\int_{\omega} A_2^{int}(x_0,U_0(x_0),\nabla U_0(x_0)):\nabla\psi\;dxdt  
\end{align*}
This implies that 
\begin{align*}
	&\;-\int_0^T\int_{\omega}Q(\alpha)(\nabla U_0+\nabla K):\nabla\psi  \;dxdt+\int_0^T\int_{\omega}Q(\alpha)\nabla U_0:\nabla\psi  \;dxdt\\&\;\quad=
	-\int_0^T\int_{\omega}Q(\alpha)\nabla U_0:\nabla\psi  \;dxdt,
\end{align*}
i.e.
\begin{align}\label{equationdeK}
	&\;\int_0^T\int_{\omega}Q(\alpha)\nabla K:\nabla\psi  \;dxdt=
	\int_0^T\int_{\omega}Q(\alpha)\nabla U_0:\nabla\psi  \;dxdt,\;\;\text{for all $\psi\in H^1(\mathbb{R}^2\times[0,T])^3$}
\end{align}
which is the equation of the corrector $K$. 
In what follows we use the following abbreviation
\begin{eqnarray*}
	&\,& a_1(\epsilon,U,P):=\int_{0}^T\int_{D}  A_1^{(\epsilon)}(U,\nabla U)\cdot P  \;dxdt,\qquad a_2(\epsilon,U,P):=\int_{0}^T\int_{D}A_2^{(\epsilon)}(U,\nabla U):\nabla P  \;dxdt\\
	&\,&\qquad\mathcal{J}(\epsilon,U):=J(\Omega_\epsilon,U,\nabla U)=\int_{0}^T\int_{D} j^{(\epsilon)}(x,U,\nabla U) \;dxdt.
\end{eqnarray*}
In this case the Lagrangian is written
\begin{eqnarray*}
	\mathcal{L}(\epsilon,U,P)=\mathcal{J}(\epsilon,U)+a_1(\epsilon,U,P)+a_2(\epsilon,U,P).
\end{eqnarray*}
It is clear that the perturbed state equation is exactly the equation (\ref{perber}). Indeed, according to the Definition \ref{defstatandadjoint}, the perturbed state equation is defined by $\partial_P\mathcal{L}(\epsilon,U_\epsilon,0)(\psi)=0$ for all $\psi\in E(\epsilon)$. So we have
\begin{eqnarray*}
	\partial_P\lag(\epsilon,U_\epsilon,0)(\psi)&=&\partial_P a_1(\epsilon,U_\epsilon,0)(\psi)+\partial_P a_2(\epsilon,U_\epsilon,0)(\psi)\\&=&a_1(\epsilon,U_\epsilon,\psi)+ a_2(\epsilon,U_\epsilon,\psi)\\&=& \int_{0}^T\int_{D}  A_1^{(\epsilon)}(U_\epsilon,\nabla U_\epsilon)\cdot \psi  \;dxdt+\int_{0}^T\int_{D}A_2^{(\epsilon)}(U_\epsilon,\nabla U_\epsilon):\nabla\psi   \;dxdt.
\end{eqnarray*} 
Then $\partial_P\lag(\epsilon,U_\epsilon,0)(\psi)=0$ implies
\begin{eqnarray*}
	\int_{0}^T\int_{D}  A_1^{(\epsilon)}(U_\epsilon,\nabla U_\epsilon)\cdot \psi  \;dxdt+\int_{0}^T\int_{D}A_2^{(\epsilon)}(U_\epsilon,\nabla U_\epsilon):\nabla\psi   \;dxdt=0.
\end{eqnarray*}
Similarly, the adjoint state equation at $\epsilon=0$ is that which satisfies $\partial_U\lag(0,U_0,P_0)(\psi)=0$. First, let us use $\partial_1\mathcal{L}$ and $\partial_2\mathcal{L}$ to denote the respective derivative of $\mathcal{L}$ with respect to the first and second arguments or simply with respect to $U$ and $\nabla U$. Then
\begin{align*}
	\partial_U\mathcal{L}(0,U_0,P_0)(\psi)&=\partial_U\mathcal{J}(0,U_0)+\partial_Ua_1(0,U_0,P_0)(\psi)+\partial_Ua_2(0,U_0,P_0)(\psi)\\&= \int_{0}^T\int_{D} \partial_U A_1^{(0)}(U_0,\nabla U_0)(\psi)\cdot P_0  \;dxdt+\int_{0}^T\int_{D}\partial_UA_2^{(0)}(U_0,\nabla U_0)(\psi):\nabla P_0  \;dxdt\\
	&= \int_{0}^T\int_{D} \partial_1 A_1^{(0)}(U_0,\nabla U_0)(\psi)\cdot P_0  \;dxdt+\int_{0}^T\int_{D}\partial_1A_2^{(0)}(U_0,\nabla U_0)(\psi):\nabla P_0  \;dxdt\\
	&+ \int_{0}^T\int_{D} \partial_2 A_1^{(0)}(U_0,\nabla U_0)(\nabla\psi)\cdot P_0  \;dxdt+\int_{0}^T\int_{D}\partial_2A_2^{(0)}(U_0,\nabla U_0)(\nabla\psi):\nabla P_0  \;dxdt.\\
\end{align*}
Hence the adjoint state equation at $\epsilon=0$ is defined by
\begin{equation}\label{adjoitogang}
	\begin{aligned}
		&\, \int_{0}^T\int_{D} \partial_1 A_1^{(0)}(U_0,\nabla U_0)(\psi)\cdot P_0  \;dxdt+\int_{0}^T\int_{D}\partial_1A_2^{(0)}(U_0,\nabla U_0)(\psi):\nabla P_0  \;dxdt\\
		&\, +\int_{0}^T\int_{D} \partial_2 A_1^{(0)}(U_0,\nabla U_0)(\nabla\psi)\cdot P_0  \;dxdt+\int_{0}^T\int_{D}\partial_2A_2^{(0)}(U_0,\nabla U_0)(\nabla\psi):\nabla P_0  \;dxdt=0.\\
	\end{aligned}
\end{equation}
Now in the following theorem we state the main result:
\begin{theorem}\label{maintheogang}
	Assume that the perturbed state equation (\ref{perber})  and the adjoint state equation (\ref{adjoitogang}) admit each a unique solution. Then, under the Hypotheses (H0a) and (Hb), the topological derivative of the objective $J(\Omega)$ is given by
	\begin{align*}
		D_TJ(\Omega)(x_0)&=\frac{1}{\vert\omega\vert}\int_{0}^T\int_{\mathbb{R}^2}\left(A_1^{\omega}(U_0(x_0),\nabla U_0(x_0)+\nabla K)- A_1^{\omega}(U_0(x_0),\nabla U_0(x_0)) \right)\cdot(P_0(x_0))  \;dxdt\\
		&-\frac{1}{\vert\omega\vert}\int_{0}^T\int_{\mathbb{R}^2}\partial_2 A_1^{\omega}(U_0(x_0),\nabla U_0(x_0))(\nabla K)\cdot(P_0(x_0))  \;dxdt\\
		&-\frac{1}{\vert\omega\vert}\int_{0}^T\int_{\omega}\left[\partial_2 A_1^{int}(U_0(x_0),\nabla U_0(x_0))\right](\nabla K)\cdot(P_0(x_0))  \;dxdt
		\\&-\frac{1}{\vert\omega\vert}\int_0^T\int_{\omega}Q(\alpha)\nabla U_0:\nabla P_0 \;dxdt\\&
		-\int_{0}^T\left[\partial_t(U_0)\cdot P_0+(\nabla\cdot(F(U_0)\cdot P_0-S(U_0)\cdot P_0\right](x_0,t)\;dt\\&-\int_{0}^T\left[ Q(\alpha)\nabla U_0:\nabla P_0\right](x_0,t)\;dt-\int_{0}^T\vert \nabla U_0\vert^2\;dt- \int_{0}^T\left(U_0(x_0,t)-U_d(x_0,t)\right)^2\;dt
	\end{align*}
	where $P_0$ is the adjoint state for $\epsilon=0$ solution to (\ref{adjoitogang}) and the corrector $K$ is solution to the  equation (\ref{equationdeK}).
\end{theorem}
\begin{proof} 
	To demonstrate this result, we directly apply the Theorem \ref{theoderivegang}. The individual terms are described in the following:
	\begin{align*}
		R_1^{A_1}(U_0,P_0)&=\lim_{\epsilon\rightarrow 0^+}\frac{1}{\vert \omega_\epsilon\vert}\int_{0}^1\left[ \partial_U a_1(\epsilon,U_0+\theta(U_\epsilon-U_0),P_0)-\partial_U a_1(\epsilon,U_0,P_0) \right](U_\epsilon-U_0) \;d\theta\\
		&=\lim_{\epsilon \to 0^+}\;\frac{1}{\vert \omega_\epsilon\vert}R_{1,\epsilon}^{A_1}.
	\end{align*}
	\begin{align*}
		R_{1,\epsilon}&=\int_{0}^1\left[ \partial_U \mathcal{L}(\epsilon,U_0+\theta(U_\epsilon-U_0),P_0)-\partial_U \mathcal{L}(\epsilon,U_0,P_0) \right](U_\epsilon-U_0) \;d\theta\\
		&=\int_{0}^1\left[ \partial_U \mathcal{J}(\epsilon,U_0+\theta(U_\epsilon-U_0),P_0)-\partial_U \mathcal{J}(\epsilon,U_0,P_0) \right](U_\epsilon-U_0) \;d\theta\\
		&+\int_{0}^1\left[ \partial_U a_1(\epsilon,U_0+\theta(U_\epsilon-U_0),P_0)-\partial_U a_1(\epsilon,U_0,P_0) \right](U_\epsilon-U_0) \;d\theta\\
		&+\int_{0}^1\left[ \partial_U a_2(\epsilon,U_0+\theta(U_\epsilon-U_0),P_0)-\partial_U a_2(\epsilon,U_0,P_0) \right](U_\epsilon-U_0) \;d\theta\\
		&=R_{1,\epsilon}^{\mathcal{J}}+R_{1,\epsilon}^{A_1}+R_{1,\epsilon}^{A_2}.
	\end{align*}
	The function $R_{2,\epsilon}$ is defined by
	\begin{align*}
		R_{2,\epsilon}&=\left( \partial_U \lag(\epsilon,U_0,P_0)-\partial_U \lag(0,U_0,P_0) \right)(U_\epsilon-U_0)\\
		&=\left( \partial_U\mathcal{J} (\epsilon,U_0,P_0)-\partial_U \mathcal{J}(0,U_0,P_0) \right)(U_\epsilon-U_0)\\
		&+\left( \partial_U a_1(\epsilon,U_0,P_0)-\partial_U a_1(0,U_0,P_0) \right)(U_\epsilon-U_0)\\
		&+\left( \partial_U a_2(\epsilon,U_0,P_0)-\partial_U a_2(0,U_0,P_0) \right)(U_\epsilon-U_0)\\
		&=R_{2,\epsilon}^{\mathcal{J}}+R_{2,\epsilon}^{A_1}+R_{2,\epsilon}^{A_2}.
	\end{align*}
	\begin{align*}
		\partial_\ell \mathcal{L}(0,U_0,P_0)&=\lim_{\epsilon \to 0^+}\frac{1}{\vert \omega_\epsilon\vert}\left(\mathcal{L}(\epsilon,U_0,P_0)-\mathcal{L}(0,U_0,P_0)\right)\\&=\lim_{\epsilon \to 0^+}\frac{1}{\vert \omega_\epsilon\vert}\left(\mathcal{J}(\epsilon,U_0,P_0)-\mathcal{J}(0,U_0,P_0)\right) +\lim_{\epsilon \to 0^+}\frac{1}{\vert \omega_\epsilon\vert}\left( a_1(\epsilon,U_0,P_0)- a_1(0,U_0,P_0) \right)
		\\
		&+\lim_{\epsilon \to 0^+}\frac{1}{\vert \omega_\epsilon\vert}\left( a_2(\epsilon,U_0,P_0)-a_2(0,U_0,P_0) \right)
		\\&=\partial_\ell \mathcal{L}^{\mathcal{J}}(U_0,P_0)+\partial_\ell \mathcal{L}^{A_1}(U_0,P_0)+\partial_\ell \mathcal{L}^{A_2}(U_0,P_0).
	\end{align*}
	\textbf{Computation of $R_1^{A_1}$}\\
	\begin{equation*}
		\begin{aligned}
			R_{1,\epsilon}^{A_1}&=\int_{0}^1\big( \partial_U a_1(\epsilon,U_0+\theta(U_\epsilon-U_0),P_0)-\partial_U a_1(\epsilon,U_0,P_0) \big)(U_\epsilon-U_0) \;d\theta\\
			&=\int_{0}^1\int_{0}^T\int_{D}\Big(\partial_1  A_1^{(\epsilon)}(U_0+\theta(U_\epsilon-U_0),\nabla U_0+\theta\nabla(U_\epsilon-U_0))\\&\qquad\qquad\qquad\qquad\qquad\qquad\qquad\qquad-\partial_1 A_1^{(\epsilon)}(U_0,\nabla U_0) \Big)(U_\epsilon-U_0)\cdot P_0 \;dxdtd\theta\\
			&+\int_{0}^1\int_{0}^T\int_{D}\Big(\partial_2  A_1^{(\epsilon)}(U_0+\theta(U_\epsilon-U_0),\nabla U_0+\theta\nabla(U_\epsilon-U_0))\\&\qquad\qquad\qquad\qquad\qquad\qquad\qquad\qquad-\partial_2 A_1^{(\epsilon)}(U_0,\nabla U_0) \Big)(\nabla(U_\epsilon-U_0))\cdot P_0 \;dxdtd\theta.
		\end{aligned}
	\end{equation*}
	Using the change of variables taking into account that $(\nabla\varphi)\circ T_\epsilon=\frac{1}{\epsilon}\nabla(\varphi\circ T_\epsilon)$ and $K_\epsilon=\frac{(U_\epsilon-U_0)\circ T_\epsilon}{\epsilon}$, we obtain
	\begin{eqnarray*}
		\begin{aligned}
			R_{1,\epsilon}^{A_1}=\epsilon^{2}\int_{0}^1\int_{0}^T\int_{D_\epsilon}&\Big(\partial_1A_1^{\omega}(U_0\circ T_\epsilon+\theta\epsilon K_\epsilon,(\nabla U_0)\circ T_\epsilon+\theta\nabla K_\epsilon)\\&\qquad\qquad-\partial_1 A_1^{\omega}(U_0\circ T_\epsilon,(\nabla U_0)\circ T_\epsilon) \Big)(\epsilon K_\epsilon)\cdot(P_0\circ T_\epsilon)  \;dxdtd\theta\\
			+\epsilon^{2}\int_{0}^1\int_{0}^T\int_{D_\epsilon}&\Big(\partial_2 A_1^{\omega}(U_0\circ T_\epsilon+\theta\epsilon K_\epsilon,(\nabla U_0)\circ T_\epsilon+\theta\nabla K_\epsilon)\\&\qquad\qquad-\partial_2A_1^{\omega}(U_0\circ T_\epsilon,(\nabla U_0)\circ T_\epsilon) \Big)(\nabla K_\epsilon)\cdot(P_0\circ T_\epsilon)  \;dxdtd\theta.\\
		\end{aligned}
	\end{eqnarray*}
	Then dividing by $\vert\omega_\epsilon\vert$ and tending $\epsilon$ to zero, we get
	\begin{equation*}
		\begin{aligned}
			R_{1}^{A_1}(U_0,P_0)=\frac{1}{\vert\omega\vert}\int_{0}^1\int_{0}^T\int_{\mathbb{R}^2}&\Big(\partial_2 A_1^{\omega}(U_0(x_0),\nabla U_0(x_0)+\theta\nabla K) \\&\qquad\qquad-\partial_2A_1^{\omega}(U_0(x_0),\nabla U_0(x_0)) \Big)(\nabla K)\cdot(P_0(x_0))  \;dxdtd\theta.\\
		\end{aligned}
	\end{equation*}
	However, we have
	\begin{align*}
		&\,\int_{0}^1\int_{0}^T\int_{\mathbb{R}^2}\partial_2 A_1^{\omega}(U_0(x_0),\nabla U_0(x_0)+\theta\nabla K)(\nabla K)(P_0(x_0))  \;dxdtd\theta\\&\,\qquad=\int_{0}^T\int_{\mathbb{R}^2}\left( A_1^{\omega}(U_0(x_0),\nabla U_0(x_0)+\nabla K)- A_1^{\omega}(U_0(x_0),\nabla U_0(x_0))\right)\cdot(P_0(x_0))  \;dxdt.
	\end{align*}
	So we get
	\begin{align*}
		R_{1}^{A_1}(U_0,P_0)&=\frac{1}{\vert\omega\vert}\int_{0}^T\int_{\mathbb{R}^2}\left(A_1^{\omega}(U_0(x_0),\nabla U_0(x_0)+\nabla K)- A_1^{\omega}(U_0(x_0),\nabla U_0(x_0)) \right)\cdot(P_0(x_0))  \;dxdt\\
		&-\frac{1}{\vert\omega\vert}\int_{0}^T\int_{\mathbb{R}^2}\partial_2 A_1^{\omega}(U_0(x_0),\nabla U_0(x_0))(\nabla K)\cdot(P_0(x_0))  \;dxdt.
	\end{align*}
	The function $R_{2,\epsilon}^{A_1}$ is defined by
	\begin{eqnarray*}
		R_{2,\epsilon}^{A_1}&=&\Big( \partial_U a_1(\epsilon,U_0,P_0)-\partial_U a_1(0,U_0,P_0) \Big)(U_\epsilon-U_0)\\
		&=&\int_{0}^T\int_{D}\left(\partial_1 A_1^{(\epsilon)}(U_0,\nabla U_0)-\partial_1 A_1^{(0)}(U_0,\nabla U_0) \right)(U_\epsilon-U_0)\cdot P_0 \;dxdt\\
		&+&\int_{0}^T\int_{D}\left(\partial_2 A_1^{(\epsilon)}(U_0,\nabla U_0)-\partial_2 A_1^{(0)}(U_0,\nabla U_0) \right)(\nabla(U_\epsilon-U_0))\cdot P_0  \;dxdt.
	\end{eqnarray*}
	The change of variables gives
	\begin{equation*}
		\begin{aligned}
			R_{2,\epsilon}^{A_1}
			=\epsilon^{2}\int_{0}^T\int_{D_\epsilon}&\Big(\partial_1 A_1^{(\epsilon)}(U_0\circ T_\epsilon,(\nabla U_0)\circ T_\epsilon)\\&\qquad\qquad\qquad-\partial_1A_1^{(0)}(U_0\circ T_\epsilon,(\nabla U_0)\circ T_\epsilon) \Big)(U_\epsilon-U_0)\circ T_\epsilon \cdot(P_0\circ T_\epsilon)  \;dxdt\\
			+\epsilon^{2}\int_{0}^T\int_{D_\epsilon}&\Big(\partial_2 A_1^{(\epsilon)}(U_0\circ T_\epsilon,(\nabla U_0)\circ T_\epsilon)\\&\qquad\qquad\qquad-\partial_2A_1^{(0)}(U_0\circ T_\epsilon,(\nabla U_0)\circ T_\epsilon) \Big)(\nabla(U_\epsilon-U_0))\circ T_\epsilon \cdot(P_0\circ T_\epsilon)  \;dxdt
			\\
			=\epsilon^{2}\int_{0}^T\int_{D_\epsilon}&\left(\partial_1 A_1^{(\epsilon)}(U_0\circ T_\epsilon,(\nabla U_0)\circ T_\epsilon)-\partial_1 A_1^{(0)}(U_0\circ T_\epsilon,(\nabla U_0)\circ T_\epsilon) \right)(\epsilon K_\epsilon)\cdot(P_0\circ T_\epsilon)  \;dxdt\\
			+\epsilon^{2}\int_{0}^T\int_{D_\epsilon}&\left(\partial_2 A_1^{(\epsilon)}(U_0\circ T_\epsilon,(\nabla U_0)\circ T_\epsilon)-\partial_2 A_1^{(0)}(U_0\circ T_\epsilon,(\nabla U_0)\circ T_\epsilon) \right)(\nabla K_\epsilon)\cdot(P_0\circ T_\epsilon)  \;dxdt.	
		\end{aligned}
	\end{equation*}
	We have the following equality
	\begin{align*}
		&\, \epsilon^{2}\int_{0}^T\int_{D_\epsilon}\left(\partial_2 A_1^{(\epsilon)}(U_0\circ T_\epsilon,(\nabla U_0)\circ T_\epsilon)-\partial_2A_1^{(0)}(U_0\circ T_\epsilon,(\nabla U_0)\circ T_\epsilon) \right)(\nabla K_\epsilon)\cdot(P_0\circ T_\epsilon)  \;dxdt\\
		&=\int_{0}^T\int_{D}\left(\partial_2 A_1^{(\epsilon)}(U_0,\nabla U_0)-\partial_2 A_1^{(0)}(U_0,\nabla U_0) \right)(\nabla K_\epsilon)\circ T^{-1}_\epsilon \cdot P_0  \;dxdt\\
		&= \int_{0}^T\int_{D}\left[\left(\chi_{\Omega_\epsilon}\partial_2 A_1^{int}+\chi_{D\backslash\Omega_\epsilon}\partial_2 A_1^{ext}-\chi_{\Omega}\partial_2 A_1^{int}-\chi_{D\backslash\Omega}\partial_2 A_1^{ext}  \right)(U_0,\nabla U_0)\right](\nabla K_\epsilon)\circ T^{-1}_\epsilon \cdot P_0 \; dxdt\\
		&=-\int_{0}^T\int_{\omega_\epsilon}\left[(\partial_2 A_1^{int}-\partial_2 A_1^{ext})(U_0,\nabla U_0)\right](\nabla K_\epsilon)\circ T^{-1}_\epsilon \cdot P_0  \;dxdt\\
		&= -\epsilon^{2}\int_{0}^T\int_{\omega}\left[(\partial_2 A_1^{int}-\partial_2 A_1^{ext})(U_0\circ T_\epsilon,(\nabla U_0)\circ T_\epsilon)\right](\nabla K_\epsilon)\cdot(P_0\circ T_\epsilon)  \;dxdt.
	\end{align*}
	Therefore $R_{2,\epsilon}^{A_1}$ becomes
	\begin{align*}
		R_{2,\epsilon}^{A_1}
		&=-\epsilon^{2}\int_{0}^T\int_{\omega}\left[(\partial_1 A_1^{int}-\partial_1 A_1^{ext})(U_0\circ T_\epsilon,(\nabla U_0)\circ T_\epsilon)\right](\epsilon K_\epsilon)\cdot(P_0\circ T_\epsilon)  \;dxdt\\
		&=-\epsilon^{2}\int_{0}^T\int_{\omega}\left[(\partial_2 A_1^{int}-\partial_2 A_1^{ext})(U_0\circ T_\epsilon,(\nabla U_0))\right](\nabla K_\epsilon)\cdot(P_0\circ T_\epsilon) \; dxdt.
	\end{align*}
	Dividing by $\vert\omega_\epsilon\vert$ and tending $\epsilon$ to zero, we get
	\begin{eqnarray*}
		R_{2}^{A_1}(U_0,P_0)=-\frac{1}{\vert\omega\vert}\int_{0}^T\int_{\omega}\left[(\partial_2 A_1^{int}-\partial_2 A_1^{ext})(U_0(x_0),\nabla U_0(x_0))\right](\nabla K)\cdot(P_0(x_0))  \;dxdt.
	\end{eqnarray*}
	Since $A_1^{ext}=0$ then
	\begin{eqnarray*}
		R_{2}^{A_1}(U_0,P_0)=-\frac{1}{\vert\omega\vert}\int_{0}^T\int_{\omega}\left[\partial_2 A_1^{int}(U_0(x_0),\nabla U_0(x_0))\right](\nabla K)\cdot(P_0(x_0))  \;dxdt.
	\end{eqnarray*}
	\begin{align*}
		\partial_\ell\mathcal{L}^{A_1}(U_0,P_0)&=\lim_{\epsilon \to 0^+}\frac{1}{\vert\omega_\epsilon\vert}\left(a_1(\epsilon,U_0,P_0)-a_1(0,U_0,P_0)\right)\\&=\lim_{\epsilon \to 0^+}\frac{1}{\vert\omega_\epsilon\vert}\int_{0}^T\int_{D}\left( A_1^{(\epsilon)}(U_0,\nabla U_0)- A_1^{(0)}(U_0,\nabla U_0)\right)\cdot P_0 \; dxdt
		\\&=-\lim_{\epsilon \to 0^+}\frac{1}{\vert\omega_\epsilon\vert}\int_{0}^T\int_{\omega_\epsilon}\left( A_1^{int}(U_0,\nabla U_0)- A_1^{ext}(U_0,\nabla U_0)\right)\cdot P_0  \;dxdt
		\\&=-\lim_{\epsilon \to 0^+}\frac{\epsilon^2}{\vert\omega_\epsilon\vert}\int_{0}^T\int_{\omega}\left( A_1^{int}(U_0\circ T_\epsilon,(\nabla U_0)\circ T_\epsilon)- A_1^{ext}(U_0\circ T_\epsilon,(\nabla U_0)\circ T_\epsilon)\right)\cdot(P_0\circ T_\epsilon) \;dxdt
		\\&=-\frac{1}{\vert\omega\vert}\int_{0}^T\int_{\omega}\left[ A_1^{int}(U_0(x_0),\nabla U_0(x_0))- A_1^{ext}(U_0(x_0),\nabla U_0(x_0))\right]\cdot(P_0(x_0)) \;dxdt
		\\&=-\int_{0}^T\left[ A_1^{int}(U_0(x_0),\nabla U_0(x_0))- A_1^{ext}(U_0(x_0),\nabla U_0(x_0))\right]\cdot(P_0(x_0)) \;dt.
	\end{align*}
	Finally we have
	\begin{eqnarray*}
		\partial_\ell\mathcal{L}^{A_1}(U_0,P_0)=-\int_{0}^T\left[\partial_t(U_0)\cdot P_0+(\nabla\cdot(F(U_0)))\cdot P_0-S(U_0)\cdot P_0\right](x_0,t)\;dt.
	\end{eqnarray*}
	\begin{equation*}
		\begin{aligned}
			R_{1,\epsilon}^{A_2}=\int_{0}^1\Big( \partial_U a_2(&\epsilon,U_0+\theta(U_\epsilon-U_0),P_0)-\partial_U a_2(\epsilon,U_0,P_0) \Big)(U_\epsilon-U_0) \;d\theta\\
			=\int_{0}^1\int_{0}^T\int_{D}&\Big(\partial_1 A_2^{(\epsilon)}(U_0+\theta(U_\epsilon-U_0,\nabla U_0+\theta\nabla(U_\epsilon-U_0))\\&\qquad\qquad\qquad\qquad\qquad\qquad-\partial_1 A_2^{(\epsilon)}(U_0,\nabla U_0) \Big)(U_\epsilon-U_0):\nabla P_0 \;dxdtd\theta\\
			+\int_{0}^1\int_{0}^T\int_{D}&\Big(\partial_2A_2^{(\epsilon)}(U_0+\theta(U_\epsilon-U_0),\nabla U_0+\theta\nabla(U_\epsilon-U_0))\\&\qquad\qquad\qquad\qquad\qquad\qquad-\partial_2A_2^{(\epsilon)}(U_0,\nabla U_0) \Big)(\nabla(U_\epsilon-U_0)):\nabla P_0 \;dxdtd\theta.
		\end{aligned}
	\end{equation*}
	Doing the same computations as before with $	R_{1}^{A_2}(U_0,P_0)=\lim_{\epsilon \to 0^+}\frac{1}{\vert\omega_\epsilon\vert}	R_{1,\epsilon}^{A_2}$, we then get
	\begin{align*}
		R_{1}^{A_2}(U_0,P_0)&=\frac{1}{\vert\omega\vert}\int_{0}^T\int_{\mathbb{R}^2}\left(A_2^{\omega}(U_0(x_0),\nabla U_0(x_0)+\nabla K)-A_2^{\omega}(U_0(x_0),\nabla U_0(x_0)) \right):\nabla P_0(x_0)  \;dxdt\\
		&-\frac{1}{\vert\omega\vert}\int_{0}^T\int_{\mathbb{R}^2}\partial_2A_2^{\omega}(U_0(x_0),\nabla U_0(x_0))(\nabla K):\nabla P_0(x_0) \;dxdt
	\end{align*}
	\begin{align*}
		R_{1}^{A_2}(U_0,P_0)=0
	\end{align*}
	\begin{align*}
		R_{2,\epsilon}^{A_2}&=\big( \partial_U a_2(\epsilon,U_0,P_0)-\partial_U a_2(0,U_0,P_0) \big)(U_\epsilon-U_0)\\
		&=\int_{0}^T\int_{D}\left(\partial_1A_2^{(\epsilon)}(U_0,\nabla U_0)-\partial_1A_2^{(0)}(U_0,\nabla U_0) \right)(U_\epsilon-U_0):\nabla P_0\;dxdt\\
		&+\int_{0}^T\int_{D}\left(\partial_2A_2^{(\epsilon)}(U_0,\nabla U_0)-\partial_2A_2^{(0)}(U_0,\nabla U_0))\right)(\nabla(U_\epsilon-U_0)):\nabla P_0 \;dxdt
	\end{align*}
	
	\begin{align*}
		R_{2}^{A_2}(U_0,P_0)=-\frac{1}{\vert\omega\vert}\int_{0}^T\int_{\omega}\left[(\partial_2A_2^{int}-\partial_2A_2^{ext})(U_0(x_0),\nabla U_0(x_0))\right](\nabla K):\nabla P_0(x_0)\; dxdt
	\end{align*}
	\begin{align*}
		R_{2}^{A_2}(U_0,P_0)=-\frac{1}{\vert\omega\vert}\;\int_0^T\int_{\omega}Q(\alpha)\nabla K:\nabla P_0 \;dxdt=
		-\frac{1}{\vert\omega\vert}\int_0^T\int_{\omega}Q(\alpha)\nabla U_0:\nabla P_0 \;dxdt,		
	\end{align*}
	\begin{align*}
		\partial_\ell\mathcal{L}^{A_2}(0,U_0,P_0)&=\lim_{\epsilon \to 0^+}\frac{1}{\vert\omega_\epsilon\vert}\left(a_2(\epsilon,U_0,P_0)-a_2(0,U_0,P_0)\right)\\&=\lim_{\epsilon \to 0^+}\frac{1}{\vert\omega_\epsilon\vert}\int_{0}^T\int_{D}\left(A_2^{(\epsilon)}(U_0,\nabla U_0)-A_2^{(0)}( U_0,\nabla U_0)\right):\nabla P_0 \;dxdt.
	\end{align*}
	Using a change of variables
	\begin{align*}
		\partial_\ell\mathcal{L}^{A_2}(0,U_0,P_0)	&=-\frac{1}{\vert\omega\vert}\int_{0}^T\int_{\omega}\left[A_2^{int}(U_0(x_0),\nabla U_0(x_0))-A_2^{ext}(U_0(x_0),\nabla U_0(x_0))\right]:\nabla P_0(x_0)\;dxdt
		\\[0.2cm] &=-\int_{0}^T\left[A_2^{int}(U_0(x_0),\nabla U_0(x_0))-A_2^{ext}(U_0(x_0),\nabla U_0(x_0))\right]:\nabla P_0(x_0)\;dt
	\end{align*}
	\begin{eqnarray*}
		\partial_\ell\mathcal{L}^{A_2}(0,U_0,P_0)=-\int_{0}^T\left[ Q(\alpha)\nabla U_0:\nabla P_0\right](x_0,t)\;dt.
	\end{eqnarray*}	
	For the objective function we have:
	\begin{eqnarray*}
		\partial_U\mathcal{J}(\epsilon,U_0)=\partial_U J(\Omega_\epsilon,U_0,\nabla U_0)(\psi)=\int_{D}\partial_1 j^{(\epsilon)}(U_0,\nabla U_0)\psi+\partial_2 j^{(\epsilon)}(U_0,\nabla U_0)\nabla\psi
	\end{eqnarray*}
	\begin{align*}
		R_{1,\epsilon}^J &=\int_{0}^1\left[\partial_U\mathcal{J}(\epsilon,U_0+\theta(U_\epsilon-U_0))-\partial_U\mathcal{J}(\epsilon,U_0) \right](U_\epsilon-U_0)\;d\theta\\
		&=\int_{0}^1\left[\partial_U J(\Omega_\epsilon,U_0+\theta(U_\epsilon-U_0),\nabla U_0+\theta\nabla(U_\epsilon-U_0))-\partial_U J(\Omega_\epsilon,U_0,\nabla U_0) \right](U_\epsilon-U_0)\;d\theta\\
		&=\int_{0}^1\int_{0}^T\int_{D}\left[\partial_1 j^{(\epsilon)}(U_0+\theta(U_\epsilon-U_0),\nabla U_0+\theta\nabla(U_\epsilon-U_0))-\partial_1 j^{(\epsilon)}(U_0,\nabla U_0) \right](U_\epsilon-U_0) \;dxdtd\theta\\
		&=\int_{0}^1\int_{0}^T\int_{D}\left[\partial_2 j^{(\epsilon)}(U_0+\theta(U_\epsilon-U_0),\nabla U_0+\theta\nabla(U_\epsilon-U_0))-\partial_2 j^{(\epsilon)}(U_0,\nabla U_0) \right](\nabla(U_\epsilon-U_0)) \;dxdtd\theta.
	\end{align*}
	A change of variables using $(\nabla\varphi)\circ T_\epsilon=\frac{1}{\epsilon}\nabla(\varphi\circ T_\epsilon)$ and $K_\epsilon=\frac{(U_\epsilon-U_0)\circ T_\epsilon}{\epsilon}$ gives
	\begin{equation*}
		\begin{aligned}
			R_{1,\epsilon}^J =\epsilon^{2}\int_{0}^1\int_{0}^T\int_{D_\epsilon}&\Big[\partial_1 j^{(\epsilon)}(U_0\circ T_\epsilon+\theta\epsilon K_\epsilon,(\nabla U_0)\circ T_\epsilon+\theta\nabla K\epsilon)\\&\qquad\qquad\qquad\qquad-\partial_1 j^{(\epsilon)}(U_0\circ T_\epsilon,(\nabla U_0)\circ T_\epsilon) \Big](\epsilon K_\epsilon)\; dxdtd\theta\\
			=\epsilon^{2}\int_{0}^1\int_{0}^T\int_{D_\epsilon}&\Big[\partial_2 j^{(\epsilon)}(U_0\circ T_\epsilon+\theta\epsilon K_\epsilon,(\nabla U_0)\circ T_\epsilon+\theta\nabla K_\epsilon)\\&\qquad\qquad\qquad\qquad-\partial_2 j^{(\epsilon)}(U_0\circ T_\epsilon,(\nabla U_0)\circ T_\epsilon) \Big](\nabla K_\epsilon) \;dxdtd\theta
		\end{aligned}
	\end{equation*}
	$R_1^{J}(U_0,P_0)=\lim_{\epsilon \to 0^+}\frac{1}{\vert\omega_\epsilon\vert}R_{1,\epsilon}^J $
	\begin{align*}
		R_1^{J}(U_0,P_0)&=\frac{1}{\vert\omega\vert}\int_{0}^1\int_{0}^T\int_{\mathbb{R}^2}\left[\partial_2 j^{\omega}(U_0(x_0),\nabla U_0(x_0)+\theta\nabla K)-\partial_2 j^{\omega}(U_0(x_0),\nabla U_0(x_0)) \right](\nabla K) \; dxdtd\theta\\
		&=\frac{1}{\vert\omega\vert}\int_{0}^T\int_{\mathbb{R}^2}\left[j^{\omega}(U_0(x_0),\nabla U_0(x_0)+\nabla K)-j^{\omega}(U_0(x_0),\nabla U_0(x_0)) \right]  \;dxdt\\
		&-\frac{1}{\vert\omega\vert}\int_{0}^T\int_{\mathbb{R}^2}\partial_2 j^{\omega}(U_0(x_0),\nabla U_0(x_0))(\nabla K) \;dxdt
	\end{align*}
	with $j^{\omega}(y_1,y_2)=\chi_\omega j^{int}(y_1,y_2)+\chi_{\mathbb{R}^2\backslash\omega}j^{ext}(y_1,y_2)$
	\begin{eqnarray*}
		R_1^{J}(U_0,P_0)&= 0
	\end{eqnarray*}
	\begin{align*}
		R_{2,\epsilon}^J&= \left[\partial_U\mathcal{J}(\epsilon,U_0)-\partial_U\mathcal{J}(0,U_0)\right](U_\epsilon-U_0)\\
		&=\int_{0}^T\int_{D}\left(\partial_1 j^{(\epsilon)}(U_0,\nabla U_0)-\partial_1 j^{(0)}(U_0,\nabla U_0)\right)(U_\epsilon-U_0) \;dxdt\\&+\int_{0}^T\int_{D}\left(\partial_2 j^{(\epsilon)}(U_0,\nabla U_0)-\partial_2 j^{(0)}(U_0,\nabla U_0)\right)(\nabla(U_\epsilon-U_0)) \;dxdt.
	\end{align*}
	Using a change of variables we get
	\begin{align*}
		R_{2,\epsilon}^J&=\epsilon^{2}\int_{0}^T\int_{D_\epsilon}\left(\partial_1 j^{(\epsilon)}(U_0\circ T_\epsilon,(\nabla U_0)\circ T_\epsilon)-\partial_1 j^{(0)}(U_0\circ T_\epsilon,(\nabla U_0)\circ T_\epsilon)\right)(\epsilon K_\epsilon)\;dxdt\\
		&+\epsilon^{2}\int_{0}^T\int_{D_\epsilon}\left(\partial_2 j^{(\epsilon)}(U_0\circ T_\epsilon,(\nabla U_0)\circ T_\epsilon)-\partial_2 j^{(0)}(U_0\circ T_\epsilon,(\nabla U_0)\circ T_\epsilon)\right)(\nabla K_\epsilon) \;dxdt\\
		&=\epsilon^{2}\int_{0}^T\int_{\omega}\left(\partial_1 j^{(\epsilon)}(U_0\circ T_\epsilon,(\nabla U_0)\circ T_\epsilon)-\partial_1 j^{(0)}(U_0\circ T_\epsilon,(\nabla U_0)\circ T_\epsilon)\right)(\epsilon K_\epsilon) \;dxdt\\
		&+\epsilon^{2}\int_{0}^T\int_{\omega}\left(\partial_2 j^{(\epsilon)}(U_0\circ T_\epsilon,(\nabla U_0)\circ T_\epsilon)-\partial_2 j^{(0)}(U_0\circ T_\epsilon,(\nabla U_0)\circ T_\epsilon)\right)(\nabla K_\epsilon) \;dxdt
	\end{align*}
	$R_2^{J}(U_0,P_0)=\lim_{\epsilon \to 0^+}\frac{1}{\vert\omega_\epsilon\vert}R_{2,\epsilon}^J $
	\begin{align*}
		R_2^{J}(U_0,P_0)&=\frac{1}{\vert\omega\vert}\int_{0}^T\int_{\omega}\left[\partial_2 j^{int}(U_0(x_0),\nabla U_0(x_0))-\partial_2 j^{ext}(U_0(x_0),\nabla U_0(x_0)) \right](\nabla K) \;dxdt
	\end{align*}
	\begin{eqnarray*}
		R_2^{J}(U_0,P_0)=0
	\end{eqnarray*}
	\begin{align*}
		\partial_\ell\mathcal{L}^J(0,U_0,P_0)&=\lim_{\epsilon \to 0^+}\frac{1}{\vert\omega_\epsilon\vert}\left[\mathcal{J}(\epsilon,U_0)-\mathcal{J}(0,U_0)\right]\\
		&=\lim_{\epsilon \to 0^+}\frac{1}{\vert\omega_\epsilon\vert}\int_{0}^T\int_{D}\left(j^{(\epsilon)}(U_0,\nabla U_0)-j^{(0)}(U_0,\nabla U_0)\right)  \;dxdt.\\
	\end{align*}
	Now we can write 
	\begin{align*}
		&\int_{0}^T\int_{D}\left(j^{(\epsilon)}(U_0,\nabla U_0)-j^{(0)}(U_0,\nabla U_0)\right) \;dxdt=\\
		&\int_{0}^T\int_{D}\left(\chi_{\Omega_\epsilon} j^{int}(U_0,\nabla U_0)+\chi_{D\backslash\Omega_\epsilon}j^{ext}(U_0,\nabla U_0)-\chi_{\Omega} j^{int}(U_0,\nabla U_0)-\chi_{D\backslash\Omega} j^{ext}(U_0,\nabla U_0)\right)  \;dxdt\\
		&=-\int_{0}^T\int_{\omega_\epsilon}\left(j^{int}(U_0,\nabla U_0)-j^{ext}(U_0,\nabla U_0) \right) \;dxdt.
	\end{align*}
	Therefore
	\begin{align*}
		\partial_\ell\mathcal{L}^J(0,U_0,P_0)&=-\lim_{\epsilon \to 0^+}\frac{1}{\vert\omega_\epsilon\vert}\int_{0}^T\int_{\omega_\epsilon}\left(j^{int}(U_0,\nabla U_0)-j^{ext}(U_0,\nabla U_0) \right) \;dxdt.
	\end{align*}
	Using a change of variables we obtain
	\begin{align*}
		\partial_\ell\mathcal{L}^J(0,U_0,P_0)&=-\lim_{\epsilon \to 0^+}\frac{\epsilon^2}{\vert\omega_\epsilon\vert}\int_{0}^T\int_{\omega}\left(j^{int}(U_0\circ T_\epsilon,(\nabla U_0)\circ T_\epsilon)-j^{ext}(U_0\circ T_\epsilon,(\nabla U_0)\circ T_\epsilon) \right) \;dxdt\\
		&=-\frac{1}{\vert\omega\vert}\int_{0}^T\int_{\omega}\left(j^{int}(U_0(x_0),\nabla U_0)(x_0)-j^{ext}(U_0(x_0),\nabla U_0(x_0)) \right) \;dxdt\\[0.2cm]
		&=-\int_{0}^Tj^{int}(U_0(x_0),\nabla U_0)(x_0))-j^{ext}(U_0(x_0),\nabla U_0(x_0)) \;dt.
	\end{align*}
	So we have well
	\begin{eqnarray*}
		\partial_\ell\mathcal{L}^J(0,U_0,P_0)=-\int_{0}^T\vert \nabla U_0\vert^2\;dt- \int_{0}^T\left(U_0(x_0,t)-U_d(x_0,t)\right)^2\;dt.
	\end{eqnarray*}
	In summary we have
	\begin{eqnarray*}
		&\,&R_1(U_0,P_0)=R_1^{A_1}(U_0,P_0)+R_1^{A_2}(U_0,P_0)+R_1^{J}(U_0,P_0)\\
		&\,&R_2(U_0,P_0)=R_2^{A_1}(U_0,P_0)+R_2^{A_2}(U_0,P_0)+R_2^{J}(U_0,P_0)\\
		&\,&\partial_\ell\mathcal{L}(0,U_0,P_0)=\partial_\ell\mathcal{L}^{A_1}(0,U_0,P_0)+\partial_\ell\mathcal{L}^{A_2}(0,U_0,P_0)+\partial_\ell\mathcal{L}^J(0,U_0,P_0).
	\end{eqnarray*}	
	Therefore
	\begin{align*}
		R_{1}(U_0,P_0)&=\frac{1}{\vert\omega\vert}\int_{0}^T\int_{\mathbb{R}^2}\left(A_1^{\omega}(U_0(x_0),\nabla U_0(x_0)+\nabla K)- A_1^{\omega}(U_0(x_0),\nabla U_0(x_0)) \right)\cdot(P_0(x_0))  \;dxdt\\
		&-\frac{1}{\vert\omega\vert}\int_{0}^T\int_{\mathbb{R}^2}\partial_2 A_1^{\omega}(U_0(x_0),\nabla U_0(x_0))(\nabla K)\cdot(P_0(x_0))  \;dxdt.
	\end{align*}
	
	\begin{align*}
		R_{2}(U_0,P_0)&=-\frac{1}{\vert\omega\vert}\int_{0}^T\int_{\omega}\left[\partial_2 A_1^{int}(U_0(x_0),\nabla U_0(x_0))\right](\nabla K)\cdot(P_0(x_0))  \;dxdt
		\\&\quad-\frac{1}{\vert\omega\vert}\int_0^T\int_{\omega}Q(\alpha)\nabla U_0:\nabla P_0 \;dxdt.
	\end{align*}
	\begin{align*}
		\partial_\ell\mathcal{L}(0,U_0,P_0)&=-\int_{0}^T\left[\partial_t(U_0)\cdot P_0+(\nabla\cdot(F(U_0)\cdot P_0-S(U_0)\cdot P_0\right](x_0,t)\;dt\\&\quad-\int_{0}^T\left[ Q(\alpha)\nabla U_0:\nabla P_0\right](x_0,t)\;dt-\int_{0}^T\vert \nabla U_0\vert^2(x_0,t)\;dt- \int_{0}^T\left((U_0-U_d)(x_0,t)\right)^2\;dt.	
	\end{align*}
	Ultimately, by combining all the above terms, we obtain the following result, giving the topological derivative of the objective function $J$ in $\Omega$ at point $x_0$ : 
	\begin{align*}
		D_TJ(\Omega)(x_0)&=\frac{1}{\vert\omega\vert}\int_{0}^T\int_{\mathbb{R}^2}\left(A_1^{\omega}(U_0(x_0),\nabla U_0(x_0)+\nabla K)- A_1^{\omega}(U_0(x_0),\nabla U_0(x_0)) \right)\cdot(P_0(x_0))  \;dxdt\\
		&-\frac{1}{\vert\omega\vert}\int_{0}^T\int_{\mathbb{R}^2}\partial_2 A_1^{\omega}(U_0(x_0),\nabla U_0(x_0))(\nabla K)\cdot(P_0(x_0))  \;dxdt\\
		&-\frac{1}{\vert\omega\vert}\int_{0}^T\int_{\omega}\left[\partial_2 A_1^{int}(U_0(x_0),\nabla U_0(x_0))\right](\nabla K)\cdot(P_0(x_0))  \;dxdt
		\\&-\frac{1}{\vert\omega\vert}\int_0^T\int_{\omega}Q(\alpha)\nabla U_0:\nabla P_0 \;dxdt\\&
		-\int_{0}^T\left[\partial_t(U_0)\cdot P_0+(\nabla\cdot(F(U_0)\cdot P_0-S(U_0)\cdot P_0\right](x_0,t)\;dt\\&-\int_{0}^T\left[ Q(\alpha)\nabla U_0:\nabla P_0\right](x_0,t)\;dt-\int_{0}^T\vert \nabla U_0\vert^2(x_0,t)\;dt- \int_{0}^T\left((U_0-U_d)(x_0,t)\right)^2\;dt.	
	\end{align*}
	This concludes the proof of the Theorem \ref{maintheogang}.
\end{proof}

\subsection*{ACKNOWLEDGEMENT}
\noindent 
This work has been supported by the Deutsche Forschungsgemeinschaft within the Priority program SPP 1962 "Non-smooth and Complementaritybased Distributed Parameter Systems: Simulation and Hierarchical Optimization". The authors would like to thank Volker Schulz ( University Trier, Trier, Germany) and Luka Schlegel (University Trier, Trier, Germany) for helpful and interesting discussions within the project Shape Optimization Mitigating Coastal Erosion (SOMICE).

	\end{document}